\documentclass[10pt,reqno]{amsart}

\usepackage{color}

\newtheorem{thm}{Theorem}[section]
\newtheorem{defn}[thm]{Definition}
\newtheorem{corollary}[thm]{Corollary}
\newtheorem{lemma}[thm]{Lemma}

\newtheorem{example}[thm]{Example}
\newtheorem{assumption}[thm]{Assumption}

\theoremstyle{remark}
\newtheorem{remark}[thm]{Remark}

\def\qed{{\hfill $\Box$ \bigskip}}

\def\XXint#1#2#3{{\setbox0=\hbox{$#1{#2#3}{\int}$}
\vcenter{\hbox{$#2#3$}}\kern-.5\wd0}}

\newcommand\aint{-\hspace{-0.38cm}\int}

\newcommand\cbrk{\text{$]$\kern-.15em$]$}}
\newcommand\opar{\text{\,\raise.2ex\hbox{${\scriptstyle
|}$}\kern-.34em$($}}
\newcommand\cpar{\text{$)$\kern-.34em\raise.2ex\hbox{${\scriptstyle |}$}}\,}

\def\<{\langle}
\def\>{\rangle}

\newcommand\bL{\mathbb{L}}

\newcommand\bE{\mathbb{E}}
\newcommand\bN{\mathbb{N}}
\newcommand\bP{\mathbb{P}}
\newcommand\bT{\mathbb{T}}
\newcommand\bZ{\mathbb{Z}}

\newcommand\fR{\mathbf{R}}

\newcommand\cF{\mathcal{F}}
\newcommand\cG{\mathcal{G}}

\newcommand\cI{\mathcal{I}}

\newcommand\cK{\mathcal{K}}
\newcommand\cL{\mathcal{L}}
\newcommand\cP{\mathcal{P}}

\newcommand\cS{\mathcal{S}}
\newcommand\cM{\mathcal{M}}
\newcommand\cT{\mathcal{T}}
\newcommand\cO{\mathcal{O}}

\def\R {{\mathbb R}}

\newcommand{\mysection}[1]{\section{#1}
\setcounter{equation}{0}}

\begin{document}

\title[$L_p$-regularity of stochastic singular integral operators]
{Boundedness  of stochastic singular integral operators
and its application to stochastic partial differential equations}

%

\author{Ildoo Kim}
\address{Center for Mathematical Challenges, Korea Institute for Advanced Study, 85 Hoegiro Dongdaemun-gu,
Seoul 130-722, Republic of Korea} \email{waldoo@kias.re.kr}
\thanks{The first author was supported by the TJ Park Science
         Fellowship of POSCO TJ Park Foundation}

\author{Kyeong-Hun Kim}
\address{Department of Mathematics, Korea University, 1 Anam-dong,
Sungbuk-gu, Seoul, 136-701, Republic of Korea}
\email{kyeonghun@korea.ac.kr}

\thanks{The second author was supported by Samsung Science  and Technology Foundation under Project Number SSTF-BA1401-02}

\subjclass[2010]{60H15, 42B20, 35S10, 35K30, 35B45}

\keywords{Stochastic Singular Integral Operator, $L_p$-estimate, Pseudo-differential operator, Stochastic Partial Differential Equation}

\begin{abstract}
In this article we  present a stochastic counterpart of the H\"ormander condition and Calder\'on-Zygmund theorem. 
Let $W_t$ be a Wiener process defined on a probability space $(\Omega,\bP)$ and $K(r,t,x,y)$ be a random kernel which is stochastically singular  in the sense that
$$
\bE  \left|\int_0^{t} \int_{|x-y|<\varepsilon}|K(s,t,x,y)|dy
dW_s\right|^p = \infty, \quad \forall \, t, p,\varepsilon>0,\, x\in \fR^d.
$$
We prove that
 the stochastic singular integral of the type
\begin{align}
                    \label{sto sin}
\bT  g(t,x) :=\int_0^{t} \int_{\fR^d} K(t,s,x,y) g(s,y)dy dW_s
\end{align}
is a   bounded  operator on $\bL_p=L_p(\Omega \times (0,\infty);
L_{p}(\fR^d))$  for any $p\geq 2$ if  it is bounded on $\bL_2$ and
the following (which we call stochastic H\"ormander condition) holds: there exists a quasi-metric $\rho$ on $(0,\infty)\times \fR^d$ and a positive constant  $C_0$ such that
 for $X=(t,x), Y=(s,y), Z=(r,z) \in (0,\infty) \times \fR^d$,
$$
\sup_{\omega\in \Omega,X,Y}\int_{0}^\infty \left[ \int_{\rho(X,Z) \geq C_0 \rho(X,Y)} | K(r,t, z,x) - K(r,s, z,y)| ~dz\right]^2 dr <\infty.
$$
 As a consequence of  our result on stochastic singular integral operators, 
we obtain the maximal regularity for a very wide class of stochastic partial differential equations.
\end{abstract}

\maketitle
\mysection{introduction}

Since Calder\'on and Zygmund's work, the singular integral theory has
been one of most important fields in Mathematics and it
 has been developed considerably in various directions 
(see e.g. \cite{grafakos2009modern,Stein1993}). In
particular,  due to H\"ormander  the singular integral 
\begin{align}
                        \label{det op}
Tf(x) : = \int_{\fR^d} K(x,y)f(y)dy
\end{align}
becomes a bounded operator on $L_p(\fR^d)$ if the kernel $K$ satisfies the H\"ormander condition
(see \cite[Theorem I.5.3]{Stein1993})
\begin{align}
                        \label{det hor}
\sup_{x,y\in \fR^d} \int_{|x-z| > 2 |x-y|} |K(x,z)-K(y,z)| dz <\infty.
\end{align}
H\"ormander's condition is considered as one of most general conditions in the theory of the singular integral, and there is a huge number of  applications to partial differential equations. For instance, consider the  heat equation
\begin{align}
                        \label{det eqn}
u_t =\Delta u + f,  \quad (t,x) \in (0,T] \times \fR^d, \quad u(0)=0.
\end{align}
 As is well known, for the solution $u$ we have 
$$
u_{x^ix^j}(t,x) = \int_0^t \int_{\fR^d}p_{x^ix^j}(s,t,x-y)f(s,y)dyds,
$$
where $p(t,x)$ is the heat kernel.  One can prove that the kernel 
$K(s,t,x,y) = 1_{s<t}p_{x^ix^j}(t-s,x-y)$ is singular but 
satisfies (\ref{det hor}) on $\fR^{d+1}$. Consequently this leads to  
\begin{align*}
\|u_{xx}\|_{L_p((0,T) \times \fR^d)} \leq C \|f\|_{L_p((0,T) \times \fR^d)}.
\end{align*}

Regarding the $L_p$-theory for stochastic PDEs,  Krylov \cite{Krylov1994, Krylov1999} firstly
introduced  the maximal $L_p$-regularity of the stochastic heat equation
\begin{align}
                    \label{sto}
du = \Delta u dt + gdW_t, \,\, t>0, \quad u(0,\cdot)=0.
\end{align}
In particular, he  proved the $L_p$-boundedness of 
\begin{align}
                        \label{625 eqn 1}
\nabla u (t,x) = \int_0^t \int_{\fR^d}\nabla p (t-s,x-y)g(s,y)dy dW_s.
\end{align}
The right hand side of (\ref{625 eqn 1}) becomes a stochastic
singular integral in the sense that
$$
\bE  \left|\int_0^{t} \int_{\fR^d}|\nabla p(s,t,x-y)|dy
dW_s\right|^p = \infty, \quad \forall \, t,p>0.
$$
Lately, $L_p$-theory has been further developed for 
  high-order stochastic PDEs,
stochastic integro-differential equations and certain stochastic
pseudo-differential equations. For related works, we refer to
\cite{kim2016parabolic, Krylov1994, Krylov1999, Mikulevicius2012} (Krylov's analytic approach) and \cite{van2012maximal, van2012stochastic} ($H^{\infty}$-calculus).   
 Krylov's approach requires differentiability of the  kernel, and  $H^{\infty}$-calculus approach  works only if the corresponding operator  is a generator of bounded analytic semigroup and does not depend on the time variable. 
 
Our primary goal is to introduce a theory with which one can  investigate the maximal regularity   for very large classes of stochastic partial differential equations. 
The stochastic singular integral of type (\ref{sto sin}) naturally appears if one tries to obtain the maximal $L_p$-regularity of solutions to stochastic partial differential equations.  We prove that the stochastic H\"ormander condition is sufficient for the $L_p$-boundedness of the stochastic integral and demonstrate that our result on stochastic singular integral  (\ref{sto sin}) leads to the maximal $L_p$-regularity of large classes of  stochastic partial differential equations.

 Here is a brief comment on our approach. 
 We noticed that  some  key techniques in Krylov's approach, e.g.   integration by parts,  are not applicable  for general 
kernels.   Hence we combined  Krylov's idea with some tools used for the deterministic singular integral theory and Calder\'on-Zygmund theorem.

The article is organized as follows. The main theorem is given Section 2 and the related parabolic
Littlewood-Paley inequality is introduced and proved in Section 3.
In section 4, the main theorem is proved on the basis of the
parabolic Littlewood-Paley inequality. Finally,  the maximal $L_p$-regularity result for  
SPDEs is given in Section 5.

We finish the introduction with  the notation used in the article.
$\bN$ and $\bZ$ denote the natural number system
and the integer number system, respectively.
As usual $\fR^{d}$
stands for the Euclidean space of points $x=(x^{1},...,x^{d})$.
 For $i=1,...,d$, multi-indices $\alpha=(\alpha_{1},...,\alpha_{d})$,
$\alpha_{i}\in\{0,1,2,...\}$, and functions $u(x)$ we set
$$
u_{x^{i}}=\frac{\partial u}{\partial x^{i}}=D_{i}u,\quad
D^{\alpha}u=D_{1}^{\alpha_{1}}\cdot...\cdot D^{\alpha_{d}}_{d}u,
\quad  \nabla u=(u_{x^1}, u_{x^2}, \cdots, u_{x^d}).
$$
We also use the notation $D^m$ for a partial derivative of order $m$ with respect to $x$.
For $p \in [1,\infty)$, a normed space $F$,
and a  measure space $(X,\mathcal{M},\mu)$, $L_{p}(X,\cM,\mu;F)$
denotes the space of all $F$-valued $\mathcal{M}^{\mu}$-measurable functions
$u$ so that
\[
\left\Vert u\right\Vert _{L_{p}(X,\cM,\mu;F)}:=\left(\int_{X}\left\Vert u(x)\right\Vert _{F}^{p}\mu(dx)\right)^{1/p}<\infty,
\]
where $\mathcal{M}^{\mu}$ denotes the completion of $\cM$ with respect to the measure $\mu$.

For $p=\infty$, we write $u \in L_{\infty}(X,\cM,\mu;F)$ iff
$$
\sup_{x}|u(x)| := \|u\|_{L_{\infty}(X,\cM,\mu;F)}
:= \inf\left\{ \nu \geq 0 : \mu( \{ x: \|u(x)\|_F > \nu\})=0\right\} <\infty.
$$
If there is no confusion for the given measure and $\sigma$-algebra, we usually omit the measure and the $\sigma$-algebra.
In particular, for a domain $\cO \subset \fR^d$ we denote $L_p(\cO) = L_p(\cO,\cL, \ell ;\fR)$ and $L_p(l_2) = L_p(\cO,\cL, \ell ;l_2)$,
where $\cL$ is the Lebesgue measurable sets, $\ell$ is the Lebesuge measure, and $l_2$ is the space of sequences $a=(a_n)$ so that
$$
|a|^2_{l_2} = \sum_{n=1}^\infty |a_n|^2 < \infty.
$$
We use  ``$:=$" to denote a definition. 
For $a,b\in \fR$, 
$a \wedge b := \min\{a,b\}$, $a \vee b := \max\{a,b\}$, and $\lfloor a \rfloor$ is the biggest integer which is less than or equal to $a$.
By $\cF$ and $\cF^{-1}$ we denote the d-dimensional Fourier transform and the inverse Fourier transform, respectively. That is,
$\cF(f)(\xi) := \int_{\fR^{d}} e^{-i x \cdot \xi} f(x) dx$ and $\cF^{-1}(f)(x) := \frac{1}{(2\pi)^d}\int_{\fR^{d}} e^{ i\xi \cdot x} f(\xi) d\xi$.
For a Lebesgue measurable set $A\subset \fR^d$, we use $|A|$ to denote its Lebesgue
measure. For a set $B$, $1_B$ is the indicator of $B$, i.e. $1_B(b)=1$ if $b \in B$ and $1_B(b)=0$ otherwise. 
For a complex number $z$, $\overline z$ is the complex conjugate of $z$ and $\Re[z]$ is the real part of $z$.
For functions  depending on $\omega$, $t$,  and $x$, the argument
$\omega \in \Omega$ will be usually omitted.
Usually $X_0$, $X$, $Y$, $Z$ denote the vectors in $(0,\infty) \times \fR^{d}$ and are represented by
$$
X_0=(t_0,x_0), \quad X=(t,x), \quad Y=(s,y), \quad Z=(r,z),
$$
where
$t_0$, $t$, $s$, $r$ are positive numberes and $x_0$, $x$, $y$, $z$ are vectors in $\fR^d$.
Finally, $N$ denotes a generic constant which can differ from line to line and
if we write $N=N(a,b,\ldots)$, then this means that the constant $N$ depends only on $a,b,\ldots$.

\mysection{main result}
Let $(\Omega,\cF,P)$ be a probability space and
$\{\cF_{t},t\geq0\}$ be an increasing filtration of
$\sigma$-fields on $\Omega$ satisfying the usual condition, i.e.
$\cF_{t}\subset\cF$ contains all $(\cF,P)$-null sets and $\cF_t = \bigcap_{s>t} \cF_s$.
 By $\cP$  we denote the predictable $\sigma$-algebra, that is,  $\cP$ is the smallest $\sigma$-algebra containing
the collection of all sets $A \times (s,t]$, where $0 \leq s \leq t <\infty$ and $A \in \cF_s$.
 Let $W^1_t, W^2_t,\cdots$  be an infinite sequence of   independent one-dimensional Wiener
processes defined on $\Omega$, each of which is a Wiener
process relative to $\{\cF_{t},t\geq0\}$.
For $T \in (0,\infty]$ and a domain $\cO \subset \fR^d$, we denote
$$
\cO_T := (0,T) \times \cO.
$$
Define
$$\bL_p(\cO_T)=L_p(\Omega \times (0,T),\cP;L_{p}(\cO)), \quad
\bL_p(\cO_T, l_2)=L_p(\Omega \times (0,T),\cP;L_{p}(\cO; l_2)),
$$
and
$$
\|g\|_{\bL_p(\cO_T, l_2)}= \left(\bE \int^{T}_0 \int_{\cO} |g|^p_{l_2} dx dt \right)^{1/p}.
$$
If $\cO_T = (0,\infty) \times \fR^d$, we simply put $\bL_p(\cO_T) =\bL_p$ and $\bL_p(\cO_T,l_2) =\bL_p(l_2)$.

Denote $\fR_+=(0,\infty)$ and let $K(r,t,z,x)=K(\omega,r,t,z,x)$ be a $\mathcal{P}\otimes \mathcal{B}(\fR_+)\otimes \mathcal{B}(\cO)\otimes  \mathcal{B}(\cO)$-measurable function such that 
 $K(r,t,z,x)=0$ if $r \geq t$.  For $g=(g^1,g^2,\ldots) \in \bL_p(\cO, l_2)$ and $(t,x) \in \cO_T$,
define
\begin{align*}
\bT_{\varepsilon}  g(t,x)
:=\int_0^{t-\varepsilon} \int_{\cO} K(r,t, z,x) g^k(r,z)dz dW_r^k
\end{align*}
and
\begin{align*}
\bT  g(t,x)
&:=\int_0^{t} \int_{\cO} K(r,t, z,x) g^k(r,z)dz dW_r^k \\
&:= \lim_{\varepsilon \downarrow 0}\bT_{\varepsilon}  g(t,x),
\end{align*}
where the sense of convergence will be specified in Assumption \ref{main as}.

\begin{defn}
       \label{defn metric}
Let $D$ be a subset of $\fR^{d+1}$.
A function $\rho(X,Y)$ defined on $D \times D$
is called a quasi-metric iff the following four properties hold:

(i) $\rho(X,Y) \geq 0$ for all $X,Y \in D $

(ii) $\rho(X,Y) = 0$ iff $X=Y$

(iii) $\rho(X,Y) = \rho(Y,X)$ for all $X,Y \in D$

(iv) There exists a constant $N_\rho \geq 1$ such that $\rho(X,Y) \leq N_\rho \left(\rho(X,Z) + \rho(Z,Y)\right)$
for all $X,Y,Z \in D$.
\end{defn}

Define balls related to the quasi-metric $\rho$ as
\begin{align*}
B_c(X):= \{ Z \in D : \rho(X,Z) < c \} , \quad X\in D,~ c>0.
\end{align*}
Note that the center $X$ of the ball $B_c(X)$ is always  in $D$.

Throughout the article we assume that  the quasi-metric $\rho$ satisfies the doubling ball condition on $D$, that is, 
for any $\gamma>0$ there exists a constant $N_\gamma$ so that
\begin{align}
                        \label{double}
| B_{\gamma c}(X)| \leq N_\gamma |B_c(X)| \quad \forall c>0,~X \in D.
\end{align}

For a locally integrable function $f$ on $D$, define its sharp function as
\begin{align*}
f^\sharp (t,x)
&:= \sup \aint_{B_c(Y)} |f(r,z) - f_{B_c(Y)}|~drdz \\
&:=\sup \frac{1}{|B_c(Y)|}\int_{B_c(Y)} |f(r,z) - f_{B_c(Y)}|~drdz\\
&\approx \sup \frac{1}{|B_c(Y)|^2}\int_{B_c(Y)} \int_{B_c(Y)} |f(Z) - f(Z')|~dZ dZ' \end{align*}
where the sup is taken over all $B_c(Y)$ containing $X=(t,x)$ and
$$
f_{B_c(Y)}=\aint_{B_c(Y)} f(r,z)~drdz.
$$
Similarly, the maximal function $\cM f(t,x)$ is defined as
$$
\cM f (t,x) := \sup \aint_{B_c(Y)} |f(r,z) |~drdz,
$$
where the sup is taken over all $B_c(Y)$ containing $X=(t,x)$.

Below is a version of  Hardy-Littlewood and Fefferman-Stein theorems.

\begin{thm}
For any $p>1$,  
\begin{align}
                    \label{as hlfs}
 \|\cM f \|_{L_p(D)}  \leq N(p)\|f \|_{L_p(D)},
\quad \forall f \in L_p(D).
\end{align}
Furthermore, if $|D|=\infty$, then
\begin{equation}
           \label{as hlfs2}
\|f \|_{L_p(D)}  \leq N_p \|f^\sharp \|_{L_p(D)}
\end{equation}
\end{thm}
For the proof of this theorem, see e.g. \cite[Theorem 2.2 and Theorem 2.4]{dongkim2016}.
If (\ref{as hlfs2}) holds, we say $\rho$ admits  the Feffreman-Stein theorem ({\bf{FS}}), which obviously holds if $|D|=\infty$.

\begin{assumption}[$L_2$-boundedness]
                    \label{main as}
For each $g \in \bL_2(\cO_T,l_2)$,
$\bT_{\varepsilon}  g(t,x)$ converges  in $\bL_2(O_T)$ as $\varepsilon \downarrow 0$.
 Moreover the operator $g \mapsto \bT g$ is bounded from $\bL_2(\cO_T,l_2)$ to $\bL_2(\cO_T)$, i.e.,
there exists a constant $N_0$ such that for all $g \in \bL_2(\cO_T,l_2)$,
\begin{align}
                        \label{l2 as}
\|\bT g\|_{\bL_2(\cO_T)} \leq N_0\|g\|_{\bL_2(\cO_T,l_2)}.
\end{align}
\end{assumption}

\begin{assumption}[A stochastic  H\"ormander condition]
                        \label{as hor}
There exist positive constants $C_0$ and $N_1$ such that for all $X=(t,x), Y=(s,y), Z=(r,z) \in \cO_T$,
\begin{align}
                            \label{hor eq}
\sup_{\omega\in \Omega}\sup_{X,Y}\int_{0}^T \left[ \int_{\rho(X,Z) \geq C_0 \rho(X,Y)} | K(r,t, z,x) - K(r,s, z,y)| ~dz\right]^2 dr
\leq N_1,
\end{align}
where $\rho$ is a quasi-metric admitting  {\bf FS}.
\end{assumption}

By $\bL(\cO_T,l_2)$ (simply $\bL(l_2)$ if $\cO_T=(0,\infty)\times \fR^d$), we denote the space of the processes $g=(g^1,g^2, \ldots)$ such that $g^k=0$ for all large $k$ and each $g^k$ is  of the type
$$
g^k(t,x)= \sum_{i=1}^{j(k)}1_{(\tau_{i-1},\tau_i]}(t) g^{ik}(x),
$$
where  $g^{ik} \in C_0^\infty(\cO)$, and $\tau_i$ are stopping times so that $\tau_i \leq T$.
It is known that $\bL(\cO_T,l_2)$ is dense in $\bL_p(\cO_T,l_2)$ for all $p \geq 1$ (for instance, see \cite[Theorem 3.10]{Krylov1999}) if $\cO_T= (0,T) \times \fR^d$.
The idea of \cite[Theorem 3.10]{Krylov1999} is easily applied even for general $\cO_T$.

Here is our main result.

\begin{thm}
                    \label{main thm}
Suppose that Assumptions \ref{main as} and \ref{as hor} hold.
Then for any $p > 2$, the operator $\bT$ can be continuously extended from $\bL(\cO_T,l_2)$ to  $\bL_p(\cO_T,l_2)$.
Moreover, for any $g \in \bL_p(\cO_T,l_2)$,
\begin{align*}
 \|\bT g\|_{\bL_p(\cO_T)} \leq N(d,p,C_0,N_0,N_1)\|g\|_{\bL_p(\cO_T,l_2)}.
\end{align*}
\end{thm}
The proof of this theorem will be given in Section \ref{pf of main thm}.

\mysection{Parabolic Littlewood-Paley inequality}

For  $l_2$-valued measurable functions $f=(f^1,f^2,\cdots)$ on $\cO_T$, denote
\begin{align}
                        \notag
\cG  f(t,x)
&:=  \left[\int_{0}^t \left|\int_{\cO}K(r,t,z,x) f(r,z)dz\right|^2_{l_2} ~dr \right]^{1/2} \\
                        \label{g definition}
&:=  \lim_{\varepsilon \downarrow 0}\left[\int_{0}^{t-\varepsilon} \left|\int_{\cO}K(r,t, z,x) f(r,z)dz\right|^2_{l_2} ~dr \right]^{1/2}.
\end{align}
In this section we study the boundedness of operator $\cG$ in $L_p(\cO_T; l_2)$. Since the integral above is deterministic one may assume that  the kernel $K$ is nonrandom throughout this section.

Theorem \ref{main thm 2} below is the main result of this section which we call  ``Parabolic Littlewood-Paley inequality". This inequality was first proved by Krylov for  $K(r,t,x,y)=\nabla_x p(t-r,x-y)$, where $p(t,x)=\frac{1}{(4\pi t)^{d/2}}e^{-|x|^2/(4t)}$ is the heat kernel.  If $f$ is independent of $t$ then parabolic Littlewood-Paley inequality with $K=\nabla_x p(t-r,x-y)$ leads to the classical (elliptic) Littlewood-Paley inequality (for instance, see \cite[Section 1]{kim2016parabolic}).
\begin{thm}[Parabolic Littlewood-Paley inequality]
                    \label{main thm 2}
Let $p \geq 2$.
Suppose  Assumptions \ref{main as} and  \ref{as hor} hold. Then for any $f \in L_2(\cO_T;l_2) \cap L_\infty(\cO_T;l_2)$,
\begin{align*}
\|\cG f\|_{L_p(\cO_T)} \leq N \|f\|_{L_p(\cO_T;l_2)},
\end{align*}
where $N$ depends only on $d$, $p$, $C_0$, $N_1$, and $N_2$.
\end{thm}
The proof of this theorem will be given at the end of this section.

\begin{lemma}
                    \label{l2 bd}
Suppose Assumption \ref{main as} holds.
Then for each $f \in L_2(\cO_T;l_2)$, $\cG f(t,x)$ is finite  almost everywhere, and moreover
 the operator $f \rightarrow \cG f$ is a bounded operator from $L_2(\cO_T;l_2)$ to $L_2(\cO_T)$.
\end{lemma}

\begin{proof}
Obviously, since $f$ is nonrandom, $f \in   \bL_2(l_2)$. By It\^o's isometry and (\ref{l2 as}),
\begin{align*}
\int_0^T \int_{\cO}|\cG  f(t,x)|^2dxdt
&=\bE\int_0^T \int_{\cO}|\bT  f(t,x)|^2dxdt \\
& \leq N_0\bE\int_0^T \int_{\cO}|f(t,x)|_{l_2}^2dxdt \\
&=N_0\int_0^T \int_{\cO}|f(t,x)|_{l_2}^2dxdt.
\end{align*}
Thus the lemma is proved.
\end{proof}

Denote
$$
\cK f(r,t,x)= \int_{\cO}K(r,t, z,x) f(r, z)dz
$$
and
\begin{align*}
Gf(t,s,x,y)=\left[\int_{0}^T  | \cK f(r,t, x)  -  \cK f(r,s,y)|_{l_2}^2 ~dr \right]^{1/2}.
\end{align*}
Observe
$$
\cG  f(t,x)
= \left[\int_{0}^t \left|\cK f(r,t,x)\right|^2_{l_2} ~dr \right]^{1/2}
$$
and
$$
Gf(t,s,x,y)=\left[\int_{0}^T  | 1_{r < t}\cK f(r,t, x)  -  1_{r < s}\cK f(r,s,y)|_{l_2}^2 ~dr \right]^{1/2},
$$
where the last equality is due to the assumption that $K(r,t,z,x)=0$ if $t \leq r$.
\begin{lemma}
                    \label{lem 1}
Let $ (t_1,x_1) \in  B_c(X_0)$ and   suppose Assumption \ref{main as} holds.
Then for any $f_1, f_2 \in L_2(\cO_T;l_2)$,
\begin{align}
                    \notag
& \aint_{B_c(X_0)} \aint_{B_c(X_0)} |\cG (f_1+f_2)(t,x)-\cG (f_1+f_2)(s,y)|~dtdxdsdy  \\
                        \label{0326 eqn 2}
&\leq 2\cM (\cG f_1) (t_1,x_1) +\aint_{B_c(X_0)} \aint_{B_c(X_0)} Gf_2(t,s,x,y) ~dt dx ds dy.
\end{align}
\end{lemma}
\begin{proof}
Set $f=f_1+f_2$ and let $(t,x),(s,y) \in B_c(X_0)$.
By Lemma  \ref{l2 bd} we may assume
$$
\cG (f_1)(t,x)+\cG (f_2)(t,x)+\cG (f_1)(s,y)+\cG (f_2)(s,y) <\infty.
$$
Then by Minkowski's inequality,
\begin{align*}
&|\cG f(t,x) - \cG f (s,y)| \\
&= \left|\left[\int_{0}^T 1_{r<t} |\cK f(r,t,x)|_{l_2}^2~dr \right]^{1/2}
-\left[\int_{0}^T 1_{r<s} |\cK f(r,s, y) |_{l_2}^2 ~dr \right]^{1/2} \right| \\
&\leq \left[\int_{0}^T  \Big|1_{r<t} \cK f(r,t,x) - 1_{r<s} \cK  f(r, s,y)\Big|_{l_2}^2 ~dr \right]^{1/2} \\
&\leq \left[\int_{0}^T 1_{r<t} |\cK f_1(r,t,x)|_{l_2}^2 ~dr \right]^{1/2}
+\left[\int_{0}^T 1_{r<s} |\cK f_1(r,s, y)|_{l_2}^2~dr \right]^{1/2} \\
&\quad +\left[\int_{0}^T  |1_{r<t} \cK f_2(r,t, x)  - 1_{r<s} \cK f_2(r,s,y)|_{l_2}^2 ~dr \right]^{1/2}.
\end{align*}
Taking mean average to the above inequality,
we get (\ref{0326 eqn 2}).
\end{proof}

Take constants $N_{\rho}$ and $C_0$ from Definition \ref{defn
metric} and Assumption \ref{as hor} respectively, and denote
$$
\gamma_0=\gamma_0(N_\rho,C_0) := (2C_0N_{\rho} +1)N_\rho.
$$
\begin{lemma}
                    \label{lem 2}
Suppose Assumption \ref{main as} holds, $f$ belongs to
$L_2(\cO_T;l_2) \cap L_\infty(\cO_T;l_2)$ and vanishes
 outside of $Q_{\gamma_0 c}(t_0,x_0)$.
Then
\begin{align}
                        \label{326 eqn 3}
\aint_{B_c(X_0)} \aint_{B_c(X_0)} Gf(t,s,x,y) ~dt dx ds dy
\leq N \|f\|_{L_\infty(\cO_T;l_2)},
\end{align}
where $N$ depends only on $d$, $\gamma_0$, and $N_0$.
\end{lemma}
\begin{proof}
Let $(t,x), (s,y) \in B_c(X_0)$ and assume $\cG f (t,x) + \cG f (s,y) <\infty$.
Then by Minkowski's inequality,
$$
Gf(t,s,x,y) \leq \cG f(t,x) + \cG f(s,y).
$$
Therefore the left side of (\ref{326 eqn 3}) is less than or equal to
$$
2\aint_{B_c(X_0)} \cG f(t,x) ~dt dx.
$$
Moreover by H\"older's inequality and Lemma \ref{l2 bd},
\begin{align*}
\aint_{B_c(X_0)} \cG f(t,x) ~dt dx
&\leq \frac{1}{|B_c(X_0)|^{1/2}}\left[\int_{B_c(X_0)} |\cG f(t,x)|^2 ~dt dx\right]^{1/2} \\
&\leq \frac{1}{|B_c(X_0)|^{1/2}}\left[\int_{\cO_T} |\cG f(t,x)|^2 ~dt dx\right]^{1/2} \\
&\leq \frac{N_0}{|B_c(X_0)|^{1/2}}\left[\int_{\cO_T} |f(t,x)|^2 ~dt dx\right]^{1/2} \\
&\leq N(d,\gamma_0,N_0) \|f\|_{L_\infty(\cO_T;l_2)},
\end{align*}
where the last inequality is due to the assumption that $f=0$ outside of $B_{\gamma_0 c}(X_0)$ and (\ref{double}).
Thus the lemma is proved.
\end{proof}

\begin{lemma}
                    \label{lem 3}
Let  $f \in L_2(\cO_T;l_2) \cap L_\infty(\cO_T;l_2)$ and $f=0 $ on $Q_{\gamma_0 c}(t_0,x_0)$.
Suppose that Assumption \ref{main as} and Assumption \ref{as hor} hold.
Then
\begin{align*}
\aint_{B_c(X_0)} \aint_{B_c(X_0)} Gf(t,s,x,y) ~dt dx ds dy \leq N \|f\|_{L_\infty(\cO_T;l_2)},
\end{align*}
where $N$ depends only on  $N_1$.
\end{lemma}
\begin{proof}

If $X=(t,x)$, $Y=(s,y) \in B_c(X_0)$, and $Z=(r,z) \in \cO_T \setminus Q_{\gamma_0 c}(t_0,x_0)$, then
\begin{align}
                        \label{326 eqn 5}
\rho(Z,X) \geq \frac{\rho(Z,X_0)}{N_\rho} - \rho(X,X_0) \geq 2C_0N_{\rho}c
\geq C_0 \rho(X,Y).
\end{align}
Thus recalling the definition of $Gf$ and the assumptions on $f$, we have
\begin{align*}
&|Gf(t,s,x,y)|^2 \\
&\leq \int_{0}^T  \left[\int_{\cO}  |K(r,t,z,x)-K(r,s,z,y)||f(r, z)|_{l_2}dz\right]^2 ~dr \\
&\leq \|f\|^2_{L_\infty(\cO_T;l_2)}\int_{0}^T  \left[\int_{A(t,r,s,x,y)}  |K(r,t,z,x)-K(r,s,z,y)|dz\right]^2 ~dr,
\end{align*}
where $A(t,r,s,x,y)$ is the set of all $z\in \fR^d$ for which inequality (\ref{326 eqn 5}) holds.
Therefore by (\ref{hor eq}),
$$
|Gf(t,s,x,y)| \leq N_1^{1/2}\|f\|_{L_\infty(\cO_T;l_2)}
$$
and
\begin{align*}
\aint_{B_c(X_0)} \aint_{B_c(X_0)} Gf(t,s,x,y) ~dt dx ds dy \leq N_1^{1/2} \|f\|_{L_\infty(\cO_T;l_2)}.
\end{align*}
The lemma is proved.
\end{proof}

\begin{lemma}
                    \label{lem 4}
Let $f_1\in L_2(\cO_T;l_2)$, $f_2 \in L_2(\cO_T;l_2) \cap L_\infty(\cO_T;l_2)$, and
 suppose Assumptions \ref{main as} and  \ref{as hor} hold.
Then for each $(t_1,x_1) \in \cO_T$,
\begin{align}
                        \label{326 eqn 7}
[\cG (f_1+f_2)]^\sharp(t_1,x_1)\leq 2\cM (\cG f_1) (t_1,x_1) + N\|f_2\|_{L_\infty(\cO_T;l_2)},
\end{align}
where $N$ depends only on $d$, $\gamma$, $N_0, C_0$, and $N_1$.
\end{lemma}
\begin{proof}
Let $(t_1,x_1) \in B_c(X_0)$. Then by Lemma \ref{lem 1},
\begin{align*}
& \aint_{B_c(X_0)} \aint_{B_c(X_0)} |\cG (f_1+f_2)(t,x)-\cG (f_1+f_2)(s,y)|~dtdxdsdy  \\
&\leq 2\cM (\cG f_1) (t_1,x_1) +\aint_{B_c(X_0)} \aint_{B_c(X_0)} Gf_2(t,s,x,y) ~dt dx ds dy.
\end{align*}
Moreover,  defining $f_{2,1}(t,x):= f_2(t,x) 1_{Q_{\gamma_0 c}(t_0,x_0)}(t,x)$ and  $f_{2,2}(t,x) := f_2(t,x) - f_{2,1}(t,x)$,
we have
\begin{align*}
&\aint_{B_c(X_0)} \aint_{B_c(X_0)} Gf_2 ~dt dx ds dy \\
&\leq \aint_{B_c(X_0)} \aint_{B_c(X_0)} Gf_{2,1} ~dt dx ds dy+\aint_{B_c(X_0)} \aint_{B_c(X_0)} Gf_{2,2} ~dt dx ds dy.
\end{align*}
Therefore we obtain (\ref{326 eqn 7}) by applying Lemma \ref{lem 2} and Lemma \ref{lem 3}.
The lemma is proved.
\end{proof}

\vspace{3mm}

{\bf Proof of Theorem \ref{main thm 2}}
\vspace{2mm}

Since the case $p=2$ is already proved in Lemma \ref{l2 bd}, we
assume $p>2$. Let $f \in L_2(\cO_T;l_2) \cap L_\infty(\cO_T;l_2)$.
For $\lambda >0$, we put
$$
f_{1,\lambda}(t,x) = f(t,x) 1_{|f| > \delta \lambda}(t,x) \quad \text{and} \quad f_{2,\lambda}(t,x) = f(t,x) 1_{|f| \leq \delta \lambda}(t,x),
$$
where $\delta$ is a positive constant  which will be specified
later. Obviously, $$f=f_{1,\lambda} + f_{2,\lambda}. $$
 Assume
$$
\lambda \leq [\cG (f)]^\sharp (t,x).
$$
Then by Lemma \ref{lem 4},
\begin{align*}
\lambda \leq [\cG (f)]^\sharp (t,x)
&\leq 2\cM (\cG f_{1,\lambda}) (t,x) + N\|f_{2,\lambda}\|_{L_\infty(\cO_T;l_2)} \\
&\leq 2\cM (\cG f_{1,\lambda}) (t,x) + N \delta \lambda,
\end{align*}
where $N$ is independent of $\lambda$ and $\delta$. Take $\delta >0$ so that $N \delta <1/2$. Then the above inequality implies that
$$
\lambda \leq  4\cM (\cG f_{1,\lambda}) (t,x).
$$
Thus
\begin{align}
                    \label{pf main 1}
|\{(t,x) \in \cO_T : \lambda \leq [\cG (f)]^\sharp (t,x)\}|
\leq |\{(t,x) \in \cO_T : \lambda \leq  4\cM (\cG f_{1,\lambda}) (t,x)\}|.
\end{align}
By (\ref{as hlfs}),
$$
\| \cG f \|_{L_p(\cO_T)} \leq N_p \| [\cG f]^\sharp \|_{L_p(\cO_T)}.
$$
Observe that
\begin{align*}
\| [\cG f]^\sharp \|^p_{L_p(\cO_T)}
=p\int_0^\infty \lambda^{p-1} |\{ (t,x) \in \cO_T : \lambda \leq [\cG (f)]^\sharp(t,x)\}|~d\lambda.
\end{align*}
Therefore by (\ref{pf main 1}), Chebyshev's inequality,  (\ref{as hlfs}), and Lemma \ref{l2 bd},
\begin{align*}
\| [\cG f]^\sharp \|^p_{L_p(\cO_T)}
&\leq p\int_0^\infty \lambda^{p-1}|\{(t,x) \in \cO_T: \lambda \leq  4\cM (\cG f_{1,\lambda}) (t,x)\}| d\lambda \\
&\leq N\int_0^\infty \lambda^{p-3}\int_{\cO_T} |\cM (\cG f_{1,\lambda})(t,x)|^2dtdx d\lambda \\
&\leq N\int_0^\infty \lambda^{p-3}\int_{\cO_T} |\cG f_{1,\lambda} (t,x)|^2 dtdx d\lambda \\
&\leq N\int_0^\infty \lambda^{p-3}\int_{\cO_T} |f_{1,\lambda}(t,x)|^2 dtdx d\lambda \\
&\leq N\int_0^\infty \lambda^{p-3}\int_{\cO_T \cap \{|f|>\delta \lambda\}} |f(t,x)|^2 dtdx d\lambda \\
&=N\int_{\cO_T} \left( \int^{|f|/\delta}_0 \lambda^{p-3}d\lambda \right)\, |f(t,x)|^{2} dtdx\\
&\leq N \|f\|_{L_p(\cO_T;l_2)}.
\end{align*}
The last inequality is due to $p>2$. The theorem is proved. \qed

\mysection{Proof of Theorem \ref{main thm}}
                                        \label{pf of main thm}

\vspace{3mm}

Let $g \in \bL(O_T)$. Then for each $\omega$, $g \in L_2(\cO_T ;l_2) \cap L_\infty(\cO_T ;l_2)$.
Therefore by Fubini's Theorem, Burkholder-Davis-Gundy's inequality, and Theorem \ref{main thm 2},
\begin{align*}
&\bE \int_0^T \|\bT g(t,\cdot)\|_{L_p}^p dt \\
&= \int_0^T \int_{\fR^d}\bE \left|\int_0^{t} \int_{\fR^d} K(r,t, z,x) g^k(r,z)dz dW_r^k\right|^p dx dt \\
&\leq N(p) \int_0^T \int_{\fR^d}\bE \left(\int_0^{t} \left|\int_{\fR^d} K(r,t, z,x) g(r,z)dz\right|^2_{l_2} dr\right)^{p/2} dx dt \\
&\leq N \bE\int_0^T \int_{\fR^d} |\cG g|^p dx dt \leq N \bE\int_0^T \int_{\fR^d} |g|^p dx dt.
\end{align*}
The theorem is proved. \qed

\mysection{ Application to SPDE: Maximal $L_p$-regularity}

We study the maximal $L_p$-regularity of   SPDEs of the type
\begin{align}
                        \label{main eqn}
du(t,x) = A(t)u(t,x)dt +\sum_{k=1}^{\infty} g^k(t,x)dW^k_t, \quad (t,x) \in (0,\infty)\times \fR^d ; \quad u(0)=0,
\end{align}
where  $W^k_t$ are  independent one-dimensional Wiener process defined on  $\Omega$.

\subsection{Time measureable pseudo-differential operator}
Assume that $A(t)$ is a pseudo differential operator with the symbol $\psi(t,\xi)$, that is,
$$
A(t)f(x) := \cF^{-1} \left[ \psi(t,\xi)  \cF(f) (\xi) \right](x).
$$
We set
$$
d_0:=\left\lfloor \frac{d}{2} \right\rfloor+1,
$$
and assume  there exists a constant $\nu>0$ such that
\begin{align}
                    \label{el con}
      \Re [-\psi(t,\xi)]  \geq \nu |\xi|^\gamma,
\end{align}
and
\begin{align}
							\notag
&\int_{R\leq |\xi|<2R} 
 \prod_{i=1}^{d_0} \left|D^{\alpha_i}_{\xi} \psi(t,\xi)\right|^{k_i}  d\xi \\
		                    \label{pseudo con}
&\quad \leq \nu^{-1} R^{\left(d+k_1(\gamma -|\alpha_1|) +k_2(\gamma -|\alpha_2|) + \cdots + k_{d_0}(\gamma -|\alpha_{d_0}|)\right)}
 \end{align}
for any $R>0$, multi-indexes  $\alpha_i\in (\bZ_+)^d$ and $k_i\in \bZ_+$  ($i=1,2,\cdots, d_0$) such that
$$
\sum_{i=1}^{d_0}|\alpha_i|+\sum_{i=1}^{d_0}k_i    \leq d_0 + \sum_{i=1}^{d_0}1_{k_i>0}.
$$
\begin{remark}
Here is a sufficient condition for  (\ref{pseudo con}):  $\exists \, c>0$ such that  
\begin{align}
                    \label{sym con}
|D^\alpha_{\xi} \psi(t,\xi)| \leq c |\xi|^{\gamma -|\alpha|}, \quad 
\quad \forall \xi \in \fR^d \setminus \{0\}, \quad \forall \,
  |\alpha| \leq  d_0.
 \end{align}
 Indeed, if this holds then, for $R\leq |\xi| < 2R$,
 $$
\prod_{i=1}^{d_0} |D^{\alpha_i}_{\xi}\psi(t,\xi)|^{k_i} \leq  N(c,d) R^{k_1(\gamma -|\alpha_1|) +k_2(\gamma -|\alpha_2|) + \cdots + k_{d_0}(\gamma -|\alpha_{d_0})|}
$$
Thus by integrating on $\{\xi \in \fR^d: R\leq |\xi| < 2R\}$ we certainly get (\ref{pseudo con}).
\end{remark}

Define
$$
p(s,t,x):=1_{0<s<t } \cF^{-1} \left[ \exp \left(\int_s^t \psi(r,\xi)dr \right) \right](x).
$$
and
$$
(-\Delta)^{\gamma/4}p(s,t,x)
:=1_{0<s<t } \cF^{-1} \left[ |\xi|^{\gamma/2}\exp \left(\int_s^t \psi(r,\xi)dr \right) \right](x).
$$
Then 
for any $g \in \bL_2(l_2)$,  the (weak) solution to \eqref{main eqn} is given by
\begin{equation}
 \label{sol}
u(t,x)=\int_0^{t} \int_{\fR^d} p(s,t, x-y) g^k(s,y)dy dW_s^k.
\end{equation}
See e.g.  \cite[Theorem 4.2]{Krylov1999} for details.  Actually  in \cite{Krylov1999} the representation formula of the weak solution is  derived only for $\psi(t,\xi)=-|\xi|^2$, but one can easily check the argument there works for the general case.

 Due to (\ref{el con}) we may say $A(t)$ is a linear operator of order $\gamma$.  Applying the Ito's formula to $|u(t,x)|^2$, taking the expectation, and then integrating over $\fR^d$, we get for any $t>0$,
 $$
 \bE \|u(t)\|^2_{L_2(\fR^d)} -2\Re\left[ \bE \int^t_0 \int_{\fR^d} u \overline{Au} dxds \right]=\bE \int^t_0 \|g\|^2_{L_2(\fR^d;l_2)} ds.
 $$
By Plancherel's theorem and (\ref{el con}),
  \begin{align*}
  &- \Re\left[\int_{\fR^d} u \overline{Au} dxds \right]
  = -\Re \left[\int_{\fR^d} \overline{\psi(s,\xi)}|\cF(u)|^2 d\xi \right]\\
  &=   \int_{\fR^d}  \Re [-\psi(s,\xi)]  |\cF(u)|^2 d\xi  \geq \nu \int_{\fR^d} |\xi|^{\gamma}|\cF(u)|^2 d\xi
  = \nu \|(-\Delta)^{\gamma/4}u\|^2_{L_2(\fR^d)}.
  \end{align*}
  It  follows that
  $$
  \bE \|u(t)\|^2_{L_2(\fR^d)} + 2\nu \bE \int^t_0\|(-\Delta)^{\gamma/4}u\|^2_{L_2(\fR^d)}ds \leq \bE \int^t_0 \|g\|^2_{L_2(\fR^d;l_2)} ds,
  $$
  and above calculations suggest that $(-\Delta)^{\gamma/4}u$ is the maximal regularity of solutions if there is no smoothness condition on $g$.

The following theorem extends  the above $L_2$-estimate to $L_p$-estimate.

\begin{thm}
                    \label{pseudo thm}
Let $p \geq 2$ and assume (\ref{el con}) and (\ref{pseudo con}) hold.  Then for any $g\in \bL(l_2)$ and $u$ defined as in (\ref{sol}), we have
\begin{align}
                            \label{apl main}
\bE \int_0^\infty \|(-\Delta)^{\gamma/4}u(t,\cdot)\|^p_{\bL_p} dt
\leq N(d,p,\gamma,\nu)\bE \int_0^\infty \|g(t,\cdot)\|^p_{L_p(l_2)}dt.
\end{align}
\end{thm}

\begin{remark}
A  proof of  (\ref{apl main})  is given  in  \cite{kim2016parabolic}  with a stronger condition than (\ref{sym con}),   that is
 \begin{align*}
 |D^\alpha_{\xi} \psi(t,\xi)| \leq \nu^{-1} |\xi|^{\gamma
 -|\alpha|}, \quad \forall\, |\alpha|\leq d_0+1.
 \end{align*}
 The proof of
\cite{kim2016parabolic}  
highly depends on the integration by parts,
which requires the stronger assumption on  $\psi(t,\xi)$.
 \end{remark}

\begin{example}
Let $m \in \bN$ and  $A(t)=(-1)^{m-1}\sum_{|\alpha|=|\beta|=m} a^{\alpha \beta}(t) D^{\alpha+\beta}$  be a $2m$-order differential operator.
Assume that  $a^{\alpha \beta}(t)$ are bounded complex-valued measurable functions
and satisfy an ellipticity condition, i.e.,
$$
\nu |\xi|^{2m} \leq  \sum_{|\alpha|=|\beta|=m}  \xi^\alpha \xi^\beta \Re\left[a^{\alpha \beta}(t)\right]
\quad \forall \xi \in \fR^{d}.
$$
Then  $A(t)$ is the pseudo-differential operator whose symbol is given by
$\psi(t,\xi)=(-1)^m \sum_{|\alpha|=|\beta|=m} a^{\alpha \beta}(t)\xi^{\alpha}\xi^{\beta}$.
Obviously $\psi(t,\xi)$ satisfies (\ref{el con}) and (\ref{sym con}) with $\gamma=2m$.
\end{example}

 \begin{example}
The class of  pseudo-differential operators we are considering in this article 
covers a certain class of non-local operators. Let $\gamma\in
(0,2)$ and denote
\begin{align*}
 A(t) u
=\int_{\fR^d \setminus \{0\}} \Big(u(t,x+y)-u(t,x)-\chi(y)(\nabla u
(t,x),y) \Big) \frac{m(t,y) }{|y|^{d+\gamma}}dy
\end{align*}
where $\chi(y)= I_{\gamma >1} + I_{|y|\leq1}I_{\gamma=1}$ and
$m(t,y)\geq 0$ is a  measurable function satisfying  the following
conditions (i)-(iv):

(i) If $\gamma=1$ then
\begin{align}           \label{cancel}
\int_{\partial B_1} m(t,w) w ~S_1(dw)=0, \quad \forall t >0,
\end{align}
where  $\partial B_{1}$ is the unit sphere in $\fR^d$ and $S_{1}(dw)$ is the surface measure on it.

(ii) The function $m=m(t,y) $ is zero-order homogeneous and differentiable in $y$ up to $d_0 = \lfloor\frac{d}{2}\rfloor+1$.

(iii) There is a constant $K$ such that for each  $t \in \fR$
\begin{eqnarray*}
         \label{mik}
\sup_{|\alpha| \leq d_0, |y|=1} |D^\alpha_y m^{(\alpha)} (t,y) | \leq K.
\end{eqnarray*}

(iv) There exists a
constant $c>0$ so that $m(t,y)>c$ on a set  $E\subset \partial B_1$
of positive  $S_1(dw)$-measure.

Using  (i)-(iv) one can check that   $A(t)$ is a pseudo differential operator with the symbol $\psi(t,\xi)$ satisfying (\ref{el con}) and (\ref{sym con}), where
\begin{align*}
\psi(t,\xi) = - c_1 \int_{\partial B_{1}} |(w,\xi)|^\gamma
[1-i\varphi^{(\gamma)}(w,\xi)] m(t,w)~S_{1}(dw),
\end{align*}
\begin{align*}
\varphi^{(\gamma)}(w,\xi)=c_2\frac{(w,\xi)}{|(w,\xi)|} I_{\gamma \neq 1}- \frac{2}{\pi} \frac{(w,\xi)}{|(w,\xi)|} \ln |(w,\xi)| I_{\gamma=1},
\end{align*}
and $c_1(\gamma,d)$, $c_2(\gamma,d)$ are certain positive constants (see  \cite{Mikulevicius2012} for the detail).
 \end{example}

To apply Theorem \ref{main thm} we set $T=\infty$, $\cO =\fR^d$, and
$$
\rho(X,Y)= |t-s|^{1/\gamma} + |x-y|,
$$
where $X=(t,x)$ and $Y=(s,y)$.
Since $\rho$ is a quasi-metric with the doubling ball condition and $|(0,\infty)\times \fR^d|=\infty$, $\rho$ admits  the Fefferman-Stein theorem. Define
\begin{align*}
\bT_{\psi,\varepsilon}g
:=\int_0^{t-\varepsilon} \int_{\fR^d} (-\Delta)^{\gamma/4}p(r,t, x-y) g^k(r,y)dy dW_r^k
\end{align*}
and
\begin{align*}
\bT_{\psi}g
:= \lim_{\varepsilon \downarrow 0}\int_0^{t-\varepsilon} \int_{\fR^d} (-\Delta)^{\gamma/4}p(r,t, x-y) g^k(r,y)dy dW_r^k,
\end{align*}
where the limit is in the sense of $\bL_2$-norm.

In the next lemma, we first show that $\bT_{\psi,\varepsilon}g$ converges with respect to the norm in $\bL_2$ and
 $\bT_{\psi}$ is a bounded operator from $\bL_2(l_2)$ to $\bL_2$.

\begin{lemma}
                    \label{a l2 lem}
For each $g \in \bL_2(l_2)$,
$\bT_{\varepsilon}  g(t,x)$ converges  in $\bL_2$ as $\varepsilon \downarrow 0$.
Moreover the operator $g \mapsto \bT g$ is bounded from $\bL_2(l_2)$ to $\bL_2$, i.e.,
there exists a constant $N_0$ such that for all $g \in \bL_2(l_2)$,
\begin{align}
                        \label{l2 est}
\bE \int_0^\infty \|\bT_{\psi} g(t,\cdot)\|_{L_2}^2 dt \leq N_0\bE \int_0^\infty \|g(t,\cdot)\|_{L_2(l_2)}^2 dt.
\end{align}
\end{lemma}
\begin{proof}
Let $\varepsilon_1> \varepsilon_2>0$. Then by Fubini's theorem, It\^o's isometry, and Plancherel's theorem,
\begin{align*}
&\bE \int_0^\infty \|\bT_{\psi,\varepsilon_1} g(t,\cdot)-\bT_{\psi,\varepsilon_1} g(t,\cdot)\|_{L_2}^2 dt  \\
&=\int_{\fR^d}\int_0^\infty \bE  \left|\int_{t-\varepsilon_1}^{t-\varepsilon_2} \int_{\fR^d} (-\Delta)^{\gamma/4}p(r,t, x-y) g^k(r,y)dy dW_r^k \right|^2 dt dx\\
&=\int_0^\infty \bE  \int_{t-\varepsilon_1}^{t-\varepsilon_2} \int_{\fR^d}\left|\int_{\fR^d} (-\Delta)^{\gamma/4}p(r,t, x-y) g^k(r,y)dy\right|_{l_2}^2  dxdr  dt\\
&=N(d)\int_0^\infty \bE  \int_{t-\varepsilon_1}^{t-\varepsilon_2}
\int_{\fR^d}  |\xi|^{\gamma}\exp \left(2\int_r^t  \Re [\psi(\rho,\xi)]d\rho \right) |\cF(g)(r,\cdot)|^2_{l_2}d\xi dr  dt\\
&\leq N(d)\bE  \int_0^\infty \int_{\fR^d} \int_{\varepsilon_2}^{\varepsilon_1}
  |\xi|^{\gamma}\exp \left(-2t\nu |\xi|^\gamma  \right) dt |\cF(g)(r,\cdot)|^2_{l_2} d\xi dr  \\
&\leq N(d)\bE  \int_0^\infty \int_{\fR^d} \left|\exp \left(-2\nu \varepsilon_1 |\xi|^\gamma  \right) -\exp \left(-2t\nu \varepsilon_2 |\xi|^\gamma  \right) \right| |\cF(g)(r,\cdot)|^2_{l_2} d\xi dr.
\end{align*}
The last term goes to zero as $\varepsilon_1,\varepsilon_2 \to 0 $ by the Lebesgue dominated convergence theorem.
Therefore $\bT_\psi g$ is well-defined and using Fubini's theorem, It\^o's isometry, and Plancherel's theorem again, we get
(\ref{l2 est}). The lemma is proved.
\end{proof}
Due to Lemma \ref{a l2 lem}, to prove (\ref{apl main}) it suffices to show that Assumption \ref{as hor} holds with
$$
K(r,t,z,x)=1_{0<r<t}(-\Delta)^{\gamma/4}p(r,t,x-z).
$$
For $0<s<t$ and $x \in \fR^d$, denote
\begin{align*}
q_1(s,t,x)= \cF^{-1} \left[ \exp \left(\int_s^t \psi(r,(t-s)^{-1/\gamma}\xi)dr \right) \right](x),
\end{align*}
and
\begin{align*}
&q_2(s,t,x) \\
&=(t-s) \cF^{-1} \left[ \psi(t,(t-s)^{-1/\gamma}\xi) |\xi|^{\gamma/2} \exp \left(\int_s^t \psi(r,(t-s)^{-1/\gamma}\xi)dr \right) \right](x).
\end{align*}
By the change of variables,
$$
(t-s)^{d/\gamma}p(s,t,(t-s)^{1/\gamma}x)=q_1(s,t,x),
$$
\begin{align}
                    \label{rela p q1}
(t-s)^{d/\gamma} (t-s)^{1/2}(-\Delta)^{\gamma/4}p(s,t,(t-s)^{1/\gamma}x)=(-\Delta)^{\gamma/4}q_1(s,t,x),
\end{align}
and
\begin{align}
                    \label{rela p q2}
\frac{\partial}{\partial t} (-\Delta)^{\gamma/4} p(s,t,x) =(t-s)^{-d/\gamma}(t-s)^{-1}q_2(s,t,(t-s)^{-1/\gamma}x).
\end{align}

\begin{lemma}
						\label{ker bd lem}
There exists a constant $N=N(d,\nu,\gamma)$ so that for any multi-index $\alpha$ with $|\alpha|\leq d_0$, $0<s<t$, and $i=1,\ldots,d$,
\begin{align*}
&\int_{\fR^d} \left|D^\alpha_{\xi}\left( |\xi|^{\gamma/2}\cF(q_1(t,s,\cdot)(\xi) \right)  \right| d\xi
+\int_{\fR^d} \left|D^\alpha_{\xi}\left( \xi^{i}|\xi|^{\gamma/2}\cF(q_1(t,s,\cdot)(\xi) \right)  \right| d\xi \\
&\quad +\int_{\fR^d} \left|D^\alpha_{\xi}\left( \cF(q_2(t,s,\cdot)(\xi) \right)  \right| d\xi \leq N.
\end{align*}
\end{lemma}
\begin{proof}
Because of the similarity, we only show 
$$
\int_{\fR^d} \left|D^\alpha_{\xi}\left( |\xi|^{\gamma/2} \cF(q_1(t,s,\cdot)(\xi) \right)  \right| d\xi \leq N.
$$
This is an easy conesequence of (\ref{el con}) and (\ref{pseudo con}).
Indeed,
\begin{align*}
&\int_{\fR^d} \left|D^\alpha_{\xi}\left( |\xi|^{\gamma/2}\cF(q_1t,s,\cdot)(\xi) \right)  \right| d\xi \\
&\leq \sum_{n \in \bZ}  \int_{2^n \leq |\xi| <2^{n+1}} \left|D^\alpha_{\xi}\left( |\xi|^{\gamma/2}\cF(q_1t,s,\cdot)(\xi) \right)  \right| d\xi \\
&= \sum_{n \in \bZ}  \int_{2^n \leq |\xi| <2^{n+1}} \left|D^\alpha_{\xi}\left( |\xi|^{\gamma/2}\exp \left(\int_s^t \psi(r,(t-s)^{-1/\gamma}\xi)dr \right) \right) \right| d\xi \\
&\leq N\sum_{n \in \bZ} \sum_{k=1}^{|\alpha|}  2^{n(d+\frac{3\gamma}{2}-k)} e^{-\nu 2^n}
\leq N(d,\nu,\gamma).
\end{align*}
The lemma is proved.
\end{proof}
Note that for any $f \in L_1(\fR^d)$,
\begin{align*}
\sup_{x \in \fR^d} \left|\cF^{-1}(f)(x)\right|  \leq N(d)\|f\|_{L_1(\fR^d)}.
\end{align*}
Thus by Lemma \ref{ker bd lem},  there exists a constant $N=N(d,\nu,\gamma)$
so that for any $t>s$ and $x \in \fR^d$
\begin{align}
                    \label{ker bds}
\left|(-\Delta)^{\gamma/4} q_1(s,t, x)\right|
+\left|\frac{\partial}{\partial x^i}(-\Delta)^{\gamma/4} q_1(s,t, x)\right|
+| q_2(s,t, x)| \leq N.
\end{align}
\begin{lemma}
						\label{ker bd lem 2}
Let $\varepsilon \in \left[0,\frac{d+3\gamma-2(d_0-1)}{2} \right)$. 
Then, there exists a constant $N=N(d,\nu,\gamma, \varepsilon)$ so that for any multi-index $\alpha$ with $|\alpha|\leq d_0-1$, $0<s<t$, and $i=1,\ldots,d$,
\begin{align*}
&\int_{\fR^d} \left||\xi|^{-\varepsilon}D^\alpha_{\xi}\left( |\xi|^{\gamma/2 }\cF(q_1(t,s,\cdot)(\xi) \right)  \right|^2 d\xi
+\int_{\fR^d} \left||\xi|^{-\varepsilon}D^\alpha_{\xi}\left( \cF(q_2(t,s,\cdot)(\xi) \right)  \right|^2 d\xi\\
&\quad+\int_{\fR^d} \left||\xi|^{-\varepsilon}D^\alpha_{\xi}\left( \xi^{i}|\xi|^{\gamma/2}\cF(q_1(t,s,\cdot)(\xi) \right)  \right|^2 d\xi \leq N.
\end{align*}
\end{lemma}
\begin{proof}
Because of the similarity, we only show 
$$
\int_{\fR^d} \left||\xi|^{-\varepsilon}D^\alpha_{\xi}\left( |\xi|^{\gamma/2} \cF(q_1(t,s,\cdot)(\xi) \right)  \right|^2 d\xi \leq N.
$$
Since $d+3\gamma -2\varepsilon-2(d_0-1) >0$,
\begin{align*}
\sum_{n -\infty}^1 2^{n(d+3\gamma-2\varepsilon-2(d_0-1))} < \infty.
\end{align*}
Therefore by (\ref{el con}) and (\ref{pseudo con}),
\begin{align*}
&\int_{\fR^d} \left||\xi|^{-\varepsilon} D^\alpha_{\xi}\left( |\xi|^{\gamma/2}\cF(q_1t,s,\cdot)(\xi) \right)  \right|^2 d\xi \\
&\leq \sum_{n \in \bZ}  \int_{2^n \leq |\xi| <2^{n+1}} \left||\xi|^{-\varepsilon} D^\alpha_{\xi}\left( |\xi|^{\gamma/2}\cF(q_1t,s,\cdot)(\xi) \right)  \right|^2 d\xi \\
&= \sum_{n \in \bZ}  \int_{2^n \leq |\xi| <2^{n+1}} \left||\xi|^{-\varepsilon} D^\alpha_{\xi}\left( |\xi|^{\gamma/2}\exp \left(\int_s^t \psi(r,(t-s)^{-1/\gamma}\xi)dr \right) \right) \right|^2 d\xi \\
&\leq N\sum_{n \in \bZ} \sum_{k=1}^{|\alpha|}  2^{n(d+3\gamma-2\varepsilon-2k)} e^{-\nu 2^n}
\leq N(d,\nu,\gamma , \varepsilon).
\end{align*}
The lemma is proved.
\end{proof}

\begin{lemma}
						\label{ker bd lem 3}
There exists a constant $N=N(d,\nu,\gamma)$ so that for all $c>0$, multi-index $|\alpha|\leq d_0$, $0<s<t$, and $i=1,\ldots,d$,
\begin{align*}
&\int_{|\xi| \geq c} \left|D^\alpha_{\xi}\left( |\xi|^{\gamma/2}\cF(q_1(t,s,\cdot)(\xi) \right)  \right|^2 d\xi
+\int_{|\xi| \geq c} \left|D^\alpha_{\xi}\left( \xi^{i}|\xi|^{\gamma/2}\cF(q_1(t,s,\cdot)(\xi) \right)  \right|^2 d\xi \\
&\quad +\int_{|\xi| \geq c} \left|D^\alpha_{\xi}\left( \cF(q_2(t,s,\cdot)(\xi) \right)  \right|^2 d\xi 
\leq N\left( 1+ 1_{c<1}c^{d+3\gamma-2d_0}\right).
\end{align*}
\end{lemma}
\begin{proof}
As in the proofs of the previous lemmas, we only show 
$$
\int_{|\xi| \geq c} \left|D^\alpha_{\xi}\left( |\xi|^{\gamma/2} \cF(q_1(t,s,\cdot)(\xi) \right)  \right|^2 d\xi \leq N.
$$
By (\ref{el con}) and (\ref{pseudo con}),
\begin{align*}
&\int_{|\xi| \geq c} \left|D^\alpha_{\xi}\left( |\xi|^{\gamma/2}\cF(q_1t,s,\cdot)(\xi) \right)  \right|^2 d\xi \\
&\leq \sum_{n \in \bZ}  \int_{2^n \leq |\xi| <2^{n+1}} 1_{|\xi| \geq c}\left|D^\alpha_{\xi}\left( |\xi|^{\gamma/2}\cF(q_1t,s,\cdot)(\xi) \right)  \right|^2 d\xi \\
&= \sum_{2^{n} \geq c/2}  \int_{2^n \leq |\xi| <2^{n+1}} 1_{|\xi| \geq c} \left|D^\alpha_{\xi}\left( |\xi|^{\gamma/2}\exp \left(\int_s^t \psi(r,(t-s)^{-1/\gamma}\xi)dr \right) \right) \right|^2 d\xi \\
&\leq N\sum_{2^{n} \geq c/2} \sum_{k=1}^{|\alpha|}  2^{n(d+3\gamma-2k)} e^{-\nu 2^n}
\leq N(d,\nu,\gamma) \left( 1+ 1_{c<1}c^{d+3\gamma-2d_0}\right).
\end{align*}
The lemma is proved.
\end{proof}

\begin{lemma}   \label{ker int fin}
Let $0< \delta < \left(\frac{\gamma}{2} \wedge \frac{1}{2} \right) $.
Then there exists a constant $N=N(d,\nu,\gamma,\delta)$ so that for any $0<s<t$
\begin{align}           \label{lemma 531}
 \int_{\fR^d} \left| |x|^{\frac{d}{2} +\delta }(-\Delta)^{\gamma/4} q_1(s,t, x) \right|^2~dx \leq N,
\end{align}
\begin{align}               \label{lemma 532}
 \int_{\fR^d} \left| |x|^{\frac{d}{2}+\delta} \frac{\partial}{\partial x^i}(-\Delta)^{\gamma/4} q_1(s,t, x) \right|^2~dx \leq N,
\end{align}
and
\begin{align}               \label{lemma 533}
\int_{\fR^d} \left| |x|^{\frac{d}{2}+\delta}q_2(s,t, x) \right|^2~dx \leq N.
\end{align}
\end{lemma}
\begin{proof}
We only prove (\ref{lemma 531}). The  proofs of (\ref{lemma 532}) and (\ref{lemma 533}) are similar.

Note that it suffices to show that for each $j=1,\ldots,d$,
\begin{align}
								\label{808 1}
 \int_{\fR^d} \left| (ix^j)^{\frac{d}{2} +\delta }(-\Delta)^{\gamma/4} q_1(s,t, x) \right|^2~dx \leq N,
\end{align}
where $i$ is the imaginary number, i.e. $i^2=-1$.
Set $\varepsilon= \varepsilon(\delta) = \frac{d}{2} +\delta - (d_0-1)$
and 
$$
\hat q(s,t,\xi) = \cF\left((-\Delta)^{\gamma/4} q_1(s,t, x) \right)(\xi).
$$
By the property of the Fourier inverse transform,
$$
(ix^j)^{d_0-1}\cF^{-1}\left(f(\xi)\right)(x)
=(-1)^{d_0-1}\cF^{-1}\left(D^{d_0-1}_{\xi^j}f(\xi)\right)(x).
$$
The left hand side of (\ref{808 1}) is equal to
\begin{align}
					\notag
& \int_{\fR^d} \left| (ix^j)^{\frac{d}{2} +\delta -(d_0-1)} \cF^{-1}\left( D_{\xi^j}^{d_0-1} \hat q(s,t,\xi)\right)(x) \right|^2~dx \\
					\label{808 2}
& \leq \int_{\fR^d} \left| |x|^{\frac{d}{2} +\delta -(d_0-1)} \cF^{-1}\left( D_{\xi^j}^{d_0-1} \hat q(s,t,\xi)\right)(x) \right|^2~dx.
\end{align}
Moreover by Plancherel's theorem, the last term above  equals to
\begin{align}
						\label{808 3}
N(d)\int_{\fR^d} \left| (-\Delta)^{\varepsilon/2} \left(D_{\xi^j}^{d_0-1} \hat q(s,t,\xi) \right) \right|^2~d\xi.
\end{align}
Obviously, $\varepsilon \in \left(0,1 \wedge \frac{d+\gamma-2(d_0-1)}{2} \right)$. Using the integral representation of the Fractional Laplacian operator 
$(-\Delta)^{\varepsilon/2}$ we get
\begin{align*}
(-\Delta)^{\varepsilon/2} (D_{\xi^j}^{d_0-1} \hat q(s,t,\xi))
=N \int_{\fR^d} \frac{ D_{\xi^j}^{d_0-1} \hat q(s,t,\xi+\eta)-D_{\xi^j}^{d_0-1} \hat q(s,t,\xi) }{|\eta|^{d+\varepsilon}} d\eta.
\end{align*}
We divide $(-\Delta)^{\varepsilon/2} (D_{\xi^j}^{d_0-1} \hat q(s,t,\xi))$ into two terms:
\begin{align*}
&N \int_{|\eta| \geq 1} \frac{ D_{\xi^j}^{d_0-1} \hat q(s,t,\xi+\eta)-D_{\xi^j}^{d_0-1} \hat q(s,t,\xi) }{|\eta|^{d+\varepsilon}} d\eta \\
&\quad+N \int_{|\eta| < 1} \frac{ D_{\xi^j}^{d_0-1} \hat q(s,t,\xi+\eta)-D_{\xi^j}^{d_0-1} \hat q(s,t,\xi) }{|\eta|^{d+\varepsilon}} d\eta
=: \cI_1(s,t,\xi) + \cI_2(s,t,\xi).
\end{align*}
By Minkowski's inequality and Lemma \ref{ker bd lem 2},
\begin{align*}
\left[\int_{\fR^d} \left|\cI_1(s,t,\xi) \right|^2 d\xi \right]^{1/2}
&\leq 2 \left\| D_{\xi^j}^{d_0-1} \hat q(s,t,\cdot)\right\|_{L_2(\fR^d)}
\int_{|\eta| \geq 1} \frac{1}{|\eta|^{d+\varepsilon}} d\eta \\
&\leq N(d,\nu,\gamma).
\end{align*}
We split $\cI_2$ into $\cI_{2,1}$, $\cI_{2,2}$, and $\cI_{2,3}$, where
\begin{align*}
\cI_{2,1}(s,t,\xi):=\int_{|\eta| < 1} 1_{|\eta| < \frac{|\xi|}{2}}\frac{ D_{\xi^j}^{d_0-1} \hat q(s,t,\xi+\eta)-D_{\xi^j}^{d_0-1} \hat q(s,t,\xi) }{|\eta|^{d+\varepsilon}} d\eta
\end{align*}
\begin{align*}
\cI_{2,2}(s,t,\xi):=\int_{|\eta| < 1} 1_{|\eta| \geq \frac{|\xi|}{2}} \frac{ D_{\xi^j}^{d_0-1} \hat q(s,t,\xi+\eta)}{|\eta|^{d+\varepsilon}} d\eta,
\end{align*}
and
\begin{align*}
\cI_{2,3}(s,t,\xi):=-\int_{|\eta| < 1} 1_{|\eta| \geq \frac{|\xi|}{2}} \frac{D_{\xi^j}^{d_0-1} \hat q(s,t,\xi) }{|\eta|^{d+\varepsilon}} d\eta.
\end{align*}
By the fundamental theorem of calculus,
\begin{align*}
|\cI_{2,1}(s,t,\xi)| 
\leq 
\int_0^1 \int_{|\eta| < 1} 1_{|\eta| < \frac{|\xi|}{2}}\frac{ \left|\nabla D_{\xi^j}^{d_0-1} \hat q(s,t,\xi+ \theta \eta) \right|}{|\eta|^{d+\varepsilon-1}} d\eta d \theta.
\end{align*}
Hence by Minkowski's inequality and Lemma \ref{ker bd lem 3}, 
\begin{align*}
\|\cI_{2,1}(s,t,\cdot)\|^2_{L_2(\fR^d)}
&\leq 
\left[ \int_{|\eta| < 1} \left(\int_{|\eta| < |\xi|}  \left|\nabla D_{\xi^j}^{d_0-1} \hat q(s,t,\xi)\right|^2d\xi \right)^{1/2}\frac{1}{|\eta|^{d+\varepsilon-1}} d\eta \right]^2 \\
&\leq 
N\left[ \int_{|\eta| < 1} \frac{1+|\eta|^{(d+3\gamma-2d_0)/2}}{|\eta|^{d+\varepsilon-1}} d\eta \right]^2
\leq N(d,\nu,\gamma)
\end{align*}
since 
$$
(d+3\gamma-2d_0)/2 -d -\varepsilon +1 >-d.
$$
On the other hand, 
if $|\xi| \geq 2$, then $\cI_{2,2}(s,t,\xi)=\cI_{2,3}(s,t,\xi)=0$ and thus we may assume $|\xi| \leq 2$.
Recalling the range of $\varepsilon$, we have
$$
\varepsilon  +\gamma  < \frac{d+3\gamma-2(d_0-1)}{2}.
$$
Hence  by H\"older's inequality and Lemma \ref{ker bd lem 2},
\begin{align*}
&|\cI_{2,2}(s,t,\xi)| \\
&\leq 
\left[\int_{|\eta| < 1} 1_{|\eta| \geq \frac{|\xi|}{2}} \frac{|\xi+\eta|^{2\varepsilon+2\gamma}}{|\eta|^{2d+2\varepsilon}} d\eta\right]^{1/2}
\left[\int_{\fR^d} \left| |\xi+\eta|^{-\varepsilon -\gamma }D_{\xi^j}^{d_0-1} \hat q(s,t,\xi+\eta) \right|^2 d\eta\right]^{1/2} \\
&\leq N
\left[\int_{|\eta| < 1} 1_{|\eta| \geq \frac{|\xi|}{2}}|\eta|^{-2d+2\gamma} d\eta\right]^{1/2}
\left[\int_{\fR^d} \left| |\eta|^{-\varepsilon-\gamma}D_{\xi^j}^{d_0-1} \hat q(s,t,\eta) \right|^2 d\eta\right]^{1/2} \\
&\leq N \left(1+|\xi|^{-\frac{d}{2} +\gamma}\right).
\end{align*}
Therefore we have
\begin{align*}
\|\cI_{2,2}(s,t,\cdot)\|^2_{L_2(\fR^d)}
\leq N \int_{|\xi| <2} \left(1+|\xi|^{-d+2\gamma}\right) d\xi \leq N(d,\nu,\gamma). 
\end{align*}
Finally by Lemma \ref{ker bd lem 2} again,
\begin{align*}
\|\cI_{2,3}(s,t,\cdot)\|^2_{L_2(\fR^d)}
\leq N.
\end{align*}
Due to (\ref{808 2}) and (\ref{808 3}), combining all estimates for $\cI_1, \cI_{2,1}, \cI_{2,2}, \cI_{2,3}$, we have
(\ref{808 1}). 
The lemma is proved. 
\end{proof}
\begin{remark}
If $\gamma$ is not small, Lemma \ref{ker int fin} is easily obtained from properties of the Fourier tansform.
Indeed,
\begin{align*}
& \int_{\fR^d} \left| |x|^{d_0} \cF^{-1}\left( D_{\xi^j}^{d_0-1} \hat q(s,t,\xi)\right)(x) \right|^2 dx \\
& \leq N\sum_{j=1}^d\int_{\fR^d} \left| (ix^j)^{d_0} \cF^{-1}\left( \hat q(s,t,\xi)\right)(x) \right| 
^2~dx \\
& = N\sum_{j=1}^d\int_{\fR^d} \left| D_{\xi^j}^{d_0} \hat q(s,t,\xi)\right|^2~d\xi.
\end{align*}
Due to (\ref{el con}) and (\ref{pseudo con}), the above term is finite if $3\gamma +d>2d_0$. 

\end{remark}

\begin{lemma}
                    \label{main cor}
Let $\delta \in \left( 0, \frac{1}{2} \wedge \frac{\gamma}{2} \right)$.
Then there exists a constant $N(d,\nu,\gamma, \delta)$ such that for all $0<s<t$, $c>0$, $a \in \fR$,
\begin{align}
                    \label{m c 1}
\int_s^t  \left|\int_{|z| \geq c} |(-\Delta)^{\gamma/4}p(r,t, z)| ~dz \right|^2 dr
\leq N\left((t-s)^{1/\gamma} c^{-1} \right)^{2\delta},
\end{align}
\begin{align}
                    \label{m c 2}
\int_{0}^a \left[\int_{\fR^d} \big|(-\Delta)^{\gamma/4}p(r,t, z+h)-(-\Delta)^{\gamma/4}p(r,t, z)\big| ~dz \right]^2 dr
\leq N \left(|h|(t - a)^{-1/\gamma}\right)^2,
\end{align}
and
\begin{align}
                        \label{m c 3}
\int_{0}^a \left[\int_{\fR^d} |(-\Delta)^{\gamma/4}p(r,t, z)-  (-\Delta)^{\gamma/4}p(r,s, z)| ~dz \right]^2dr
\leq N\left((t-s)(s -a)^{-1}\right)^2.
\end{align}
\end{lemma}
\begin{proof}
First we prove (\ref{m c 1}).
By (\ref{rela p q1}), H\"older's inequality, and (\ref{lemma 531}),
\begin{align*}
&\left|\int_{|z| \geq c} |(-\Delta)^{\gamma/4}p(r,t, z)| ~dz \right|^2 \\
&=(t-r)^{-1} \left|\int_{(t-r)^{1/\gamma}|z| \geq c} |(-\Delta)^{\gamma/4} q_1(r,t, z)| ~dz \right|^2 \\
&\leq (t-r)^{-1}
\int_{(t-r)^{1/\gamma}|z| \geq c}  |z|^{-d-2\delta}  ~dz \int_{(t-r)^{1/\gamma}|z| \geq c} \left| |z|^{\frac{d}{2}+\delta} (-\Delta)^{\gamma/4} q_1(r,t, z)\right|^2 ~dz  \\
&\leq (t-r)^{-1+(2\delta)/\gamma} c^{-2\delta}.
\end{align*}
Hence we have
\begin{align*}
\int_s^t \left|\int_{|z| \geq c} (-\Delta)^{\gamma/4}p(r,t, z) ~dz \right|^2 dr
&\leq N  \left((t-s)^{1/\gamma} c^{-1} \right)^{2\delta}.
\end{align*}

Next we prove (\ref{m c 2}). From (\ref{rela p q1}),
\begin{align*}
\frac{\partial}{\partial x^i}(-\Delta)^{\gamma/4}p(r,t, z)
=(t-r)^{-d/\gamma}(t-r)^{-1/2-1/\gamma}\frac{\partial}{\partial x^i} (-\Delta)^{\gamma/4}q_1(r,t, (t-r)^{-1/\gamma}z),
\end{align*}
and by H\"older's inequality, (\ref{ker bds}), and (\ref{lemma  532}),
\begin{align*}
& \left[\int_{\fR^d} \Big|\frac{\partial}{\partial x^i} \Delta^{\gamma/2}q_1(r,t, z) \Big| dz \right]^2 \\
&\leq  N+\int_{|z| \geq 1} |z|^{-d-2\delta} dz \int_{\fR^d} \left| |z|^{\frac{d}{2}+\delta} \frac{\partial}{\partial x^i} \Delta^{\gamma/2}q_1(r,t, z) \right|^2 dz \leq N,
\end{align*}
where $N$ is independent of $t$ and $r$.
Therefore, by the fundamental theorem of calculus,
\begin{align*}
&\int_{0}^a \left[\int_{\fR^d} \big| (-\Delta)^{\gamma/4}p(r,t, z+h)-(-\Delta)^{\gamma/4}p(r,t, z) \big| ~dz \right]^2 dr  \\
&\leq |h|^2\int_{0}^a \left[\int_{\fR^d} |\nabla \Delta^{\gamma/2}p(r,t, z)| ~dz\right]^2 dr  \\
&\leq |h|^2\int_{0}^a (t-r)^{-1-2/\gamma} \left(\int_{\fR^d} \Big|\frac{\partial}{\partial x^i} (-\Delta)^{\gamma/4}q_1(r,t, z) \Big|~dz\right)^2 dr \\
&\leq N|h|^2\int_{0}^a (t-r)^{-1-2/\gamma} dr \leq N|h|\int_{t-a}^\infty r^{-1-2/\gamma} dr=N\left(|h|(t - a)^{-1/\gamma}\right)^2.
\end{align*}

It only remains to prove (\ref{m c 3}).
By the mean-value theorem and (\ref{rela p q2}),
\begin{align*}
&|(-\Delta)^{\gamma/4}p(r,t, z)-  (-\Delta)^{\gamma/4}p(r,s, z)| \\
&\leq |t-s|(\theta t + (1-\theta)s-r)^{-d/\gamma -3/2}\left|q_2(r, \theta t + (1-\theta)s, (\theta t + (1-\theta)s-r)^{-1/\gamma}z)\right|,
\end{align*}
where $\theta \in [0,1]$.
Moreover by H\"older's inequality, (\ref{ker bds}), and (\ref{lemma 533}),
$$
\int_{\fR^d} | q_2 (r, \theta t + (1-\theta)s, z)|~dz <N,
$$
where $N$ is independent of $t$, $s$, $r$, and $\theta$.
Therefore,
\begin{align*}
&\int_{0}^a \left[\int_{\fR^d} \left|(-\Delta)^{\gamma/4}p(r,t, z)-  (-\Delta)^{\gamma/4}p(r,s, z)\right| ~dz\right]^2 dr\\
&\leq \int_{0}^a \frac{|t-s|^2}{\big(\theta t + (1-\theta)s -r \big)^3} dr \leq |t-s|^2(s-a)^{-2}.
\end{align*}
The lemma is proved.
\end{proof}
In the following corollary, we finally prove that  the kerenel
$$K(r,t,z,x):= 1_{0<r<t}(-\Delta)^{\gamma/4}p(r,t, x-z)$$ satisfies Assumption
\ref{as hor}.
Recall
$$
\rho(X,Y)= |t-s|^{1/\gamma} + |x-y|.
$$
For $r >0$ and $X=(t,x), Y=(s,y)\in (0,\infty)\times \fR^{d}$, set
\begin{align*}
A(r,X,Y)
&:= \left\{z \in \fR^d: \rho(X,Z) \geq 4\cdot 2^{1/\gamma}\rho(X,Y) \right\} \\
&= \left\{z \in \fR^d: |t-r|^{1/\gamma} +|x-z| \geq 4 \cdot 2^{1/\gamma}(|t-s|^{1/\gamma} + |x-y|) \right\},
\end{align*}
where $Z=(r,z)$.
\begin{corollary}
There is a constant $N=N(d,\nu,\gamma)$ so that 
for any $X=(t,x), Y=(s,y) \in (0,\infty) \times \fR^d$ and $r>0$,
\begin{align*}
&\int_{0}^\infty \left[ \int_{A(r,X,Y)} \left|1_{0<r<t}(-\Delta)^{\gamma/4}p(r,t,x-z) - 1_{0<r<s}(-\Delta)^{\gamma/4}p(r,s, y-z)\right| ~dz \right]^2 dr \\
&\leq N.
\end{align*}

\end{corollary}
\begin{proof}
We use the notation $(-\Delta)^{\gamma/4}p(r,t,x)$
instead of $1_{0<r<t}(-\Delta)^{\gamma/4}p(r,t,x)$.
In other words, we assume that $(-\Delta)^{\gamma/4}p(r,t,x)=0$ unless $0<r<t$. 
Moreover  we may assume $t\geq s$ without loss of generality.  
Since the proof of the case $t=s$ is simpler, we only prove the case $t>s$.

Fix a constant $\delta \in \left( 0, \frac{1}{2} \wedge \frac{\gamma}{2} \right)$. 
Denote
\begin{align*}
\cI(r,X,Y)  
&=\left[ \int_{A(r,X,Y)} \left| (-\Delta)^{\gamma/4}p(r,t,x-z) - (-\Delta)^{\gamma/4}p(r,s, y-z)\right| ~dz\right]^2.
\end{align*}
Obviously $\cI(r,X,Y)=0$ if $r\geq t$.  Thus
\begin{align*}
\int_{0}^\infty \cI(r,X,Y) dr
&=\int_{2s-t}^t \cI(r,X,Y) dr +\int_{0}^{2s-t} \cI(r,X,Y) dr \\
&=:\cI_1(X,Y) + \cI_2(X,Y).
\end{align*}
First we estimate $\cI_1(X,Y)$.
By (\ref{m c 1}),
\begin{align*}
\cI_1(X,Y)
&\leq \int_{2s-t}^t\left[ \int_{|z| \geq |t-s|^{1/\gamma} } \left| (-\Delta)^{\gamma/4}p(r,t,z) \right| ~dz\right]^2dr \\
&\quad+\int_{2s-t}^s\left[ \int_{|z| \geq |t-s|^{1/\gamma} } \left| (-\Delta)^{\gamma/4}p(r,s,z) \right| ~dz\right]^2dr
\leq N.
\end{align*}
We split $\cI_2$.
Observe
\begin{align*}
&\cI_2  
\leq \cI_{2,1}+ \cI_{2,2}\\
&:= \int_{0}^{2s-t} \left[ \int_{A(r,X,Y)} \left|(-\Delta)^{\gamma/4}p(r,t,x-z) - (-\Delta)^{\gamma/4}p(r,t, y-z)\right| ~dz\right]^2 dr  \\
&+\int_{0}^{2s-t} \left[ \int_{A(r,X,Y)} \left|(-\Delta)^{\gamma/4}p(r,t,y-z) - (-\Delta)^{\gamma/4}p(r,s, y-z)\right| ~dz\right]^2 dr.
\end{align*}
If $|x-y| \leq (t-s)^{1/\gamma}$ then by (\ref{m c 2}),
\begin{align*}
\cI_{2,1} \leq N\left(|x-y| (t-s)^{-1/\gamma}\right)^2 \leq N.
\end{align*}
On the other hand, if $|x-y| > (t-s)^{1/\gamma}$, then
\begin{align*}
\cI_{2,1}\leq 2\cI_{2,1,1}+\cI_{2,1,2},
\end{align*}
where
\begin{align*}
&\cI_{2,1,1}:= \int_{s-|x-y|^\gamma}^{t} \left[ \int_{|z| \geq |t-s|^{1/\gamma} + |x-y|} \left|(-\Delta)^{\gamma/4}p(r,t,z)\right| ~dz\right]^2 dr,
\end{align*}
and
\begin{align*}
&\cI_{2,1,2}  \\
&:= \int_{0}^{s-|x-y|^\gamma} \left[ \int_{\fR^d} \left|(-\Delta)^{\gamma/4}p(r,t,x-z) - (-\Delta)^{\gamma/4}p(r,t, y-z)\right| ~dz\right]^2 dr.
\end{align*}
By (\ref{m c 1}) again,
$$
\cI_{2,1,1} \leq N \left(\left(t-s + |x-y|^\gamma\right)^{-1/\gamma}\left((t-s)^{1/\gamma}+|x-y|\right)\right)^{2\delta} \leq N
$$
and by (\ref{m c 2})
\begin{align*}
\cI_{2,1,2} \left(|x-y|\left(t - s +|x-y|^\gamma \right)^{-1/\gamma}\right)^2 \leq N.
\end{align*}
It only remains to estimate $\cI_{2,2}$. However, this is an easy consequence of (\ref{m c 3})
since $2s-t <t$.
Indeed,
\begin{align*}
\cI_{2,2} \leq N\left((t-s)|t-s|^{-1}\right)^2 \leq N.
\end{align*}
The corollary is proved.
\end{proof}

Finally,  applying Theorem \ref{main thm} with
$$
\bT_{\varepsilon} = \bT_{\psi, \varepsilon} \quad \text{and} \quad  \bT = \bT_{\psi},
$$
we obtain (\ref{apl main}).

\subsection{Infinitesimal generators of  subordinate Brownian motions}

In this subsection we consider the infinitesimal generators of subordinate Brownian motions. In general the symbols of such operators do not satisfy (\ref{el con}) which is assumed in the previous subsection.

Let $S_t$  be a subordinator, that is, an increasing L\'evy process taking values in $[0,\infty)$ with $S_0=0$. A subordinator $S$ is completely characterized by its Laplace exponent $\phi$, i.e. 
$\bE e^{-\lambda S_t}=e^{-t \phi(\lambda)}$ for $\lambda>0$. Actually a function function $\phi:(0,\infty)\to (0,\infty)$
with $\phi(0+)=0$ is  a Laplace exponent of a subordinator if and only if it is a Bernstein function (i.e. $(-1)^nD^n \phi \leq 0, \forall n$). Also  it is   of
the form
\begin{equation*}\label{e:bernstein-function}
\phi(\lambda)=b \lambda +\int_{(0,\infty)}(1-e^{-\lambda t})\, \mu(dt)\, ,\quad \lambda >0\, ,
\end{equation*}
where $b\ge 0$ and $\mu$ is a measure on $(0,\infty)$ satisfying $\int_{(0,\infty)}(1\wedge t)\, \mu(dt)<\infty$, called the L\'evy measure.   Let $B_t$ be a $d$-dimensional Brownian motion independent of $S_t$. Then $\phi(\Delta)$ can be defined as the infinitesimal generator of the subordinate Brownian motion $B_{S_t}$:
$$
\phi(\Delta)f(x)=\lim_{t\to 0} \frac{\bE f(x+B_{S_t})-f(x)}{t}, \quad f\in C^2_b(\fR^d),
$$
and its integral version is 
\begin{equation}\label{e:phirep}
b \Delta f(x)+\int_{\R^d}\left(f(x+y)-f(x)-\nabla f(x) \cdot y {\mathbf 1}_{\{|y|\le 1\}}\right)\, J(y)\, dy \, ,
\end{equation}
where $J(x)=j(|x|)$ with $j:(0,\infty)\to (0,\infty)$ given by
$$
j(r)=\int_0^{\infty} (4\pi t)^{-d/2} e^{-r^2/(4t)}\, \mu(dt)\, .
$$
 See e.g. \cite{schilling2012bernstein} for more details. In general for $n=0,1,2,\cdots$, we define 
 $\phi(\Delta)^{n/2}$ on the Schwartz space $\cS$ as the operator with symbol $-\phi(|\xi|^2)$, i.e.
$$
\phi(\Delta)^{n/2}f (x) 
:=-\phi(-\Delta)^{n/2}f (x) 
:= \cF^{-1} \left[ -\phi(|\xi|^2)^{n/2} \cF f(\xi) \right](x).
$$
Consider the operator $A(t) = \phi(\Delta)$. Then (\ref{main eqn}) has a solution $u$ given by
$$
u(t,x)=\int_0^{t} \int_{\fR^d} p(t-s, x-y) g^k(s,y)dy dW_s^k,
$$
where
$$
p(t,x)= \cF^{-1} \left[ \exp \left(- t\phi(|\xi|^2)\right) \right](x).
$$
Let $\phi^{-1}$ denote the generalized inverse of $\phi$, i.e.
$$
\phi^{-1}(t) := \inf\{ s > 0: \phi(s) \geq t \}.
$$

\begin{assumption}
 \label{ass berstein}
(i)  There exists a constant $N$ such that for all $t\leq T\in (0,\infty]$ and
$x\in \fR^d$
\begin{align}
                        \label{as ker}
\left|\phi(\Delta)^{1/2} p(t, \cdot)(x) \right|
 \leq   N \left(t^{-1/2}   (\phi^{-1}(t^{-1}))^{d/2}  \wedge \frac{\phi(| x|^{-2})^{1/2}}{ | x|^{d}}\right),
\end{align}
\begin{align}
                        \label{as ker 2}
\left|\phi(\Delta)^{1/2} \nabla p(t, \cdot)(x) \right|
 \leq   N \left(t^{-1/2}   (\phi^{-1}(t^{-1}))^{(d+1)/2}  \wedge \frac{\phi(| x|^{-2})^{1/2}}{ | x|^{d+1}}\right),
\end{align}
and
\begin{align}
                        \label{as ker 3}
\left|\phi(\Delta)^{3/2} p(t, \cdot)(x) \right|
 \leq   N \left(t^{-3/2}   (\phi^{-1}(t^{-1}))^{d/2}  \wedge t^{-1}\frac{\phi(| x|^{-2})^{1/2}}{ | x|^{d}}\right).
\end{align}

(ii) $\phi$ satisfies the following scaling property: there exist
positive constants $N_1$, $N_2$, $\delta_1$, and $\delta_2$ so that
\begin{align}
                    \label{as sca}
N_1 \left(\frac{b}{a}\right)^{\delta_1} \leq \frac{\phi(b)}{\phi(a)}
\leq N_2\left(\frac{b}{a}\right)^{\delta_2}, \quad \quad \forall
\,\, 0<a\leq b.
\end{align}

\end{assumption}

Using (\ref{as sca}) one can find  constants $\bar \delta_1$, $\bar
\delta_2$, $\bar N_1$, and $\bar N_2$ 
depending only on $\delta_1$,
$\delta_2$, $N_1$, and $N_2$ 
so that 
\begin{align}
                        \label{phi inverse}
\bar N_1 \left(\frac{b}{a}\right)^{\bar \delta_1} \leq
\frac{\phi^{-1}(b)}{\phi^{-1}(a)} \leq \bar N_2\left(\frac{b}{a}
\right)^{\bar \delta_2}, \quad \quad 0<a\leq b.
\end{align}
 Furthermore, 
$$
\lim_{t \uparrow \infty} \phi(t)=\lim_{t \uparrow \infty} \phi^{-1}(t)=\infty, \quad  \quad
\lim_{t \downarrow 0} \phi(t)=\lim_{t \downarrow 0} \phi^{-1}(t)=0,
$$
and
\begin{align}
                        \label{at}
 \phi( \phi^{-1}(t))=t.
\end{align}

\begin{example}
A sufficient condition to (\ref{as ker})-(\ref{as sca})  can be founded e.g.  in
\cite{kim2013parabolic}:

\noindent (H1): $\exists$ constants $0<\delta_1\le \delta_2
<1$ and $c_1, c_2>0$  such that
\begin{equation*}\label{e:H1}
c_1\lambda^{\delta_1} \phi(t) \le \phi(\lambda t) \le c_2
\lambda^{\delta_2} \phi(t), \quad \lambda \ge 1, t \ge 1\, ;
\end{equation*}
\noindent (H2): $\exists$  constants $0<\delta_3 \le 1$ and
$c_3>0$  such that
\begin{equation*}\label{e:H2}
 \phi(\lambda t) \le c_3 \lambda^{\delta_3} \phi(t), \quad \lambda \le 1, t \le 1\, .
\end{equation*}
Actually using (H1) and (H2) one can prove (see \cite{kim2013parabolic})
\begin{align*}
\left|\phi(\Delta)^{n/2} D^\beta p(t, \cdot)(x) \right|
 \leq   N \left(t^{-n/2}   (\phi^{-1}(t^{-1}))^{(d+|\beta|)/2}  \wedge t^{-(n-1)/2}\frac{\phi(| x|^{-2})^{1/2}}{ | x|^{d+|\beta|}}\right)
\end{align*}
for any $n\leq 3$ and multi-index $\beta$ with $|\beta| \leq 2$.

Here are some examples of Bernstein functions
satisfying (H1) and (H2):
\begin{itemize}
\item[(1)] $\phi(\lambda)=\lambda^\alpha + \lambda^\beta$, $0<\alpha<\beta<1$;
\item[(2)] $\phi(\lambda)=(\lambda+\lambda^\alpha)^\beta$, $\alpha, \beta\in (0, 1)$;
\item[(3)] $\phi(\lambda)=\lambda^\alpha(\log(1+\lambda))^\beta$, $\alpha\in (0, 1)$,
$\beta\in (0, 1-\alpha)$;
\item[(4)] $\phi(\lambda)=\lambda^\alpha(\log(1+\lambda))^{-\beta}$, $\alpha\in (0, 1)$,
$\beta\in (0, \alpha)$;
\item[(5)] $\phi(\lambda)=(\log(\cosh(\sqrt{\lambda})))^\alpha$, $\alpha\in (0, 1)$;
\item[(6)] $\phi(\lambda)=(\log(\sinh(\sqrt{\lambda}))-\log\sqrt{\lambda})^\alpha$, $\alpha\in (0, 1)$.
\end{itemize}
For example, the subordinate Brownian motion corresponding to the
example (1) $\phi(\lambda)=\lambda^\alpha + \lambda^\beta$ is the
sum of two independent symmetric $\alpha$ and $\beta$ stable
processes, and its infinitesimal generator is
$-(-\Delta)^{\beta/2}-(-\Delta)^{\alpha/2}$.
\end{example}

Define
\begin{align*}
\bT_{\phi,\varepsilon}g
:=\int_0^{t-\varepsilon} \int_{\fR^d} \phi(\Delta)^{1/2}p(t-r, x-y) g^k(r,y)dy dW_r^k
\end{align*}
and
\begin{align*}
\bT_{\phi}g
:= \lim_{\varepsilon \downarrow 0}\int_0^{t-\varepsilon} \int_{\fR^d} \phi(\Delta)^{1/2}p(t-r, x-y) g^k(r,y)dy dW_r^k,
\end{align*}
where the limit is in the sense of $\bL_2$-norm.

Here is the main result of this subsection.
\begin{thm}
         \label{thm berstein}
         Let $p\in [2,\infty)$ and  Assumption \ref{ass berstein} hold. Then
         \begin{align*}
 \|\bT_{\phi} g\|_{\bL_p(\cO_T)} \leq N\|g\|_{\bL_p(\cO_T,l_2)} \quad \forall g \in \bL_p(\cO_T,l_2),
\end{align*}       
where $N$ depends only on $d$ and the constants appearing in Assumption \ref{ass berstein}.
         \end{thm}
%

To apply Theorem \ref{main thm}, we set $\cO =\fR^d$ and
$$
\rho(X,Y)= \left(\phi^{-1}\left(|t-s|^{-1} \right) \right)^{-1/2} + |x-y|,
$$
where $X=(t,x)$, $Y=(s,y)$, and $ \left(\phi^{-1}\left(0^{-1} \right) \right)^{-1/2}:=0$.
Due to (\ref{phi inverse}), one can easily check that $\rho$ is a quasi-metric and satisfies the doubling ball condition.
Thus, we only need to check that
$$
K(t-r,z,x):= 1_{0 < r<t <T}\phi(\Delta)^{1/2}p(t-r, x-z)
$$
satisfies Assumption
\ref{as hor}
since the proof of $L_2$-boundedness of $\cT_{\phi}g$ can be easily proved as  the proof of  Lemma \ref{a l2 lem}.

\begin{lemma}
There exists a constant $N$ such that for all $0<a<s<t<T$, $c>0$,
\begin{align}
                    \label{615 1}
\int_s^t  \left[ \int_{ |z| \geq c} \left|\phi (\Delta)^{1/2} p(t-r, z)\right| ~dz \right]^2 dr
\leq N (t-s)\phi(c^{-2}),
\end{align}
\begin{align}
                        \notag
&\int_{0}^a \left[\int_{\fR^d} \left|\phi(\Delta)^{1/2}p(t-r, z+h)-\phi(\Delta)^{1/2}p(t-r, z)\right| ~dz \right]^2 dr \\
                        \label{615 2}
&\leq N|h|^2 \phi^{-1} \left( (t-a)^{-1} \right),
\end{align}
and
\begin{align}
                    \label{615 3}
\int_{0}^a \left[\int_{\fR^d} \left|\phi(\Delta)^{1/2}p(t-r, z)-  \phi(\Delta)^{1/2}p(s-r, z)\right| ~dz \right]^2dr
\leq N(t-s)^2(s -a)^{-2},
\end{align}
where $N$ depends only on $d$ and the constants appearing in Assumption \ref{ass berstein}.
\end{lemma}
\begin{proof}
First we prove (\ref{615 1}).
This is an easy consequence of assumptions on $\phi$.
Indeed, by (\ref{as ker}) and (\ref{as sca}),
\begin{align*}
\int_{ |z| \geq c} |\phi (\Delta)^{1/2} p(t-r, z)| ~dz
&\leq \int_{ |z| \geq c} \frac{\phi(|z|^{-2})^{1/2}}{ |z|^{d}} ~dz  \\
&\leq \int_{ |z| \geq 1} \frac{\phi(|cz|^{-2})^{1/2}}{ |z|^{d}} ~dz  \\
&\leq N \phi(c^{-2})^{1/2}.
\end{align*}
Next we prove (\ref{615 2}).
By the fundamental theorem of calculus,
\begin{align*}
&\int_{\fR^d} \left|\phi(\Delta)^{1/2}p(t-r, z+h)-\phi(\Delta)^{1/2}p(t-r, z)\right| ~dz \\
& \leq |h|\int_{\fR^d} \big|\phi(\Delta)^{1/2}\nabla p(t-r, z)dz.
\end{align*}
Denote
$$
a_t := \left(\phi^{-1}(t^{-1}) \right)^{1/2}.
$$
Then by (\ref{as ker 2}), (\ref{as sca}), and (\ref{at}),
\begin{align*}
&\int_{\fR^d} \left|\phi(\Delta)^{1/2}\nabla p(t-r, z) \right|dz \\
&= \int_{|z| < 1/ a_{t-r}} \left|\phi(\Delta)^{1/2} \nabla p(t-r, z) \right|dz
+ \int_{|z| \geq 1/a_{t-r}} \left|\phi(\Delta)^{1/2} \nabla p(t-r, z) \right|dz \\
&\leq N\left( (1/a_{t-r})^{d} (t-r)^{-1/2}   (\phi^{-1}((t-r)^{-1}))^{(d+1)/2}
+\int_{|z| \geq 1/a_{t-r}}  \frac{\phi(|z|^{-2})^{1/2}}{|z|^{d+1}} dx \right)\\
&\leq N\left( (t-r)^{-1/2}   (\phi^{-1}((t-r)^{-1}))^{1/2} \right).
\end{align*}
Hence by (\ref{phi inverse}),
\begin{align*}
&\int_{0}^a \left[\int_{\fR^d} \big|\phi(\Delta)^{1/2}p(t-r, z+h)-\phi(\Delta)^{1/2}p(t-r, z)\big| ~dz \right]^2 dr  \\
& \leq N |h|^2 \left(  \int_0^a   (t-r)^{-1}\phi^{-1}\left((t-r)^{-1}\right) dr\right) \\
& \leq N |h|^2 \left(  \phi^{-1} \left( (t-a)^{-1} \right) \right).
\end{align*}
Finally we prove (\ref{615 3}).
By the mean-value theorem, (\ref{as ker 3}), and (\ref{as sca}),
\begin{align*}
&\int_{\fR^d} \left|\phi(\Delta)^{1/2}p(t-r, z)-  \phi(\Delta)^{1/2}p(s-r, z) \right| ~dz \\
&=(t-s)\int_{\fR^d} \left|\frac{d}{dr}\phi(\Delta)^{1/2}p(\theta t+ (1-\theta)s - r, z) \right|~dz \\
&=(t-s)\int_{\fR^d} \left|\phi(\Delta)^{3/2}p(\theta t+ (1-\theta)s - r, z) \right|~dz \\
&\leq N(t-s)(\theta t + (1-\theta)s -r)^{-3/2}  \\
&\quad + N(t-s)\int_{|z| \geq 1/a_{\theta t + (1-\theta)s -r }} (\theta t + (1-\theta)s -r)^{-1}\frac{\phi(|z|^{-2})^{1/2}}{ |z|^{d}} dz \\
&\leq N(t-s)(\theta t + (1-\theta)s -r)^{-3/2},
\end{align*}
where $\theta \in [0,1]$.
Therefore,
\begin{align*}
&\int_{0}^a \left[\int_{\fR^d} |\phi(\Delta)^{1/2}p(t-r, z)-  \phi(\Delta)^{1/2}p(s-r, z)| ~dz \right]^2dr \\
&\leq N (t-s)^2 \int_{0}^a (\theta t + (1-\theta)s -r)^{-3} dr \leq (t-s)^2 (s-a)^{-2}.
\end{align*}
The lemma is proved.
\end{proof}

In the following corollary, we finally prove that  the kerenel
$$
K(t-r,z,x):= 1_{0 < r<t <T}\phi(\Delta)^{1/2}p(t-r, x-z)
$$
satisfies Assumption \ref{as hor}.
Recall
$$
\rho(X,Y)= \left(\phi^{-1}\left(|t-s|^{-1} \right) \right)^{-1/2} + |x-y|.
$$
Due to (\ref{phi inverse}), there exists a constant $N_\phi \geq 1$ so that
$$
\phi^{-1}\left( a \right)
\leq N_\phi \phi^{-1}\left( 2^{-1} a \right) \qquad \forall a > 0.
$$
For $r>0$, $X=(t,x), Y=(s,y)\in (0,\infty)\times \fR^d$ set
\begin{align*}
A(r,X,Y) 
:= \left\{z \in \fR^d: \rho(X,Z) \geq 4N_\phi \rho(X,Y) \right\},
\end{align*}
where $Z=(r,z)$.
For the notational convenience, we use
$$
\phi(\Delta)^{1/2}p(t-r, x-z)
$$
to denote
$$
1_{0 < r<t <T}\phi(\Delta)^{1/2}p(t-r, x-z),
$$
that is, we assume
$$
\phi(\Delta)^{1/2}p(t-r, x-z)=0
$$
unless $0<r<t<T$.
\begin{corollary}
There exists a constant $N$ so that for all $X=(t,x), Y=(s,y) \in (0,T) \times \fR^d$,
\begin{align*}
&\int_{0}^T \left[ \int_{A(r,X,Y)} \left|\phi(\Delta)^{1/2}p(t-r,x-z) - \phi(\Delta)^{1/2}p(s-r, y-z)\right| ~dz\right]^2 dr
\leq N,
\end{align*}
where $N$ depends only on $d$ and the constants appearing in Assumption \ref{ass berstein}.
\end{corollary}
\begin{proof}
Without loss of generality, we assume $t \geq s$. 
We only focus on proving the case $t>s$ since the proof of the case $t=s$ is simpler. 
Denote
\begin{align*}
\cI(r,X,Y)
=\left[ \int_{A(r,X,Y)} \left| \phi(\Delta)^{1/2}p(t-r,x-z) - \phi(\Delta)^{1/2}p(s-r, y-z)\right| ~dz\right]^2.
\end{align*}
If $r \geq t$, then $\cI(r,X,Y)=0$. Thus
\begin{align*}
\int_{0}^T \cI(r,X,Y) dr
&=\int_{2s-t}^t \cI(r,X,Y) dr +\int_{0}^{2s-t} \cI(r,X,Y) dr \\
&=:\cI_1(X,Y) + \cI_2(X,Y).
\end{align*}
First we estimate $\cI_1(X,Y)$.
By (\ref{615 1}),
\begin{align*}
&\cI_1(X,Y) \\
&\leq \int_{2s-t}^t\left[ \int_{|z| \geq \left(\phi^{-1}(|t-s|^{-1})\right)^{-1/2} } \left| \phi(\Delta)^{1/2}p(t-r,z) \right| ~dz\right]^2dr \\
&\quad +\int_{2s-t}^s\left[ \int_{|z| \geq \left(\phi^{-1}(|t-s|^{-1})\right)^{-1/2} } \left| \phi(\Delta)^{1/2}p(s-r,z) \right| ~dz\right]^2dr
\leq N.
\end{align*}
We split $\cI_2$.
Observe
\begin{align*}
\cI_2
&\leq \cI_{2,1}+ \cI_{2,2}\\
&:= \int_{0}^{2s-t} \left[ \int_{A(r,X,Y)} \left|\phi(\Delta)^{1/2}p(t-r,x-z) - \phi(\Delta)^{1/2}p(t-r, y-z)\right| ~dz\right]^2 dr  \\
&+\int_{0}^{2s-t} \left[ \int_{A(r,X,Y)} \left|\phi(\Delta)^{1/2}p(t-r,y-z) -\phi(\Delta)^{1/2}p(s-r, y-z)\right| ~dz\right]^2 dr.
\end{align*}
If $|x-y| \leq \left(\phi^{-1}(|t-s|^{-1})\right)^{-1/2}$ then by (\ref{615 2}),
\begin{align*}
\cI_{2,1}
\leq N|x-y|^2 \left(  \phi^{-1} \left( |t-s|^{-1} \right) \right)
\leq N.
\end{align*}
On the other hand, if
\begin{align*}
|x-y| > \left(\phi^{-1}(|t-s|^{-1})\right)^{-1/2},
\end{align*}
then
\begin{align*}
\cI_{2,1}\leq 2\cI_{2,1,1}+\cI_{2,1,2},
\end{align*}
where
\begin{align*}
&\cI_{2,1,1}:= \int_{s- \left(\phi\left(|x-y|^{-2}\right)\right)^{-1}}^{t} \left[ \int_{|z| \geq  \left(\phi^{-1}(|t-s|^{-1})\right)^{-1/2} + |x-y|} \left| \phi(\Delta)^{1/2}p(t-r,z)\right| ~dz\right]^2 dr
\end{align*}
\begin{align*}
&\cI_{2,1,2} \\
&:= \int_{0}^{s-\left(\phi\left(|x-y|^{-2}\right)\right)^{-1}} \left[ \int_{\fR^d} 
\left|\phi(\Delta)^{1/2}p(t-r,x-z) - \phi(\Delta)^{1/2}p(t-r, y-z)\right| ~dz\right]^2 dr.
\end{align*}
By (\ref{615 1}),
\begin{align*}
\cI_{2,1,1}
&\leq N \left[\left(t-s + \left(\phi\left(|x-y|^{-2}\right)\right)^{-1}\right) \phi \left( \left( \left(\phi^{-1}(|t-s|^{-1})\right)^{-1/2}+|x-y|\right)^{-2}\right)\right]  \\
&\leq N
\end{align*}
and by (\ref{615 2}),
\begin{align*}
\cI_{2,1,2}
&\leq N|x-y|^2 \phi^{-1} \left( \left(t-s +\left(\phi\left(|x-y|^{-2}\right)\right)^{-1}\right)^{-1} \right)\leq N.
\end{align*}
It only remains to estimate $\cI_{2,2}$. However, this is an easy consequence of (\ref{615 3}).
Indeed,
\begin{align*}
\cI_{2,2} \leq N\left((t-s)|t-s|^{-1}\right)^2 \leq N.
\end{align*}
The corollary is proved.
\end{proof}
Consequently,  to prove  Theorem \ref{thm berstein} it is enough to apply Theorem \ref{main thm} with
$$
\bT_{\varepsilon} = \bT_{\phi, \varepsilon} \quad \text{and} \quad  \bT = \bT_{\phi}.
$$

\end{document}